\documentclass[11pt]{amsart}
\usepackage{amsmath,amssymb,amsthm}\allowdisplaybreaks[4]
\usepackage{url}
\usepackage{graphicx}   
\usepackage{verbatim}       
\usepackage{layout}
\usepackage{galois}
\usepackage{mathrsfs,amsfonts}
\usepackage{color}
\usepackage{bbm}
\usepackage{enumerate}
\usepackage[round]{natbib}
\usepackage{hyperref}
\usepackage[title]{appendix}
\newcommand{\ind}{1{\hskip -2.5 pt}\hbox{I}}

\usepackage{geometry}
\numberwithin{equation}{section}

\newtheorem{theorem}{Theorem}[section]
\newtheorem{lemma}[theorem]{Lemma}
\newtheorem{proposition}[theorem]{Proposition}

\newtheorem{definition}[theorem]{Definition}

\def\KK{\mathbb{K}}

\def\cP{\mathcal{P}}

\def\al{{\alpha}}\def\be{{\beta}}\def\de{{\delta}}
\def\ep{{\epsilon}}\def\ga{{\gamma}}
\def\la{{\lambda}}

\def\<{\left<}\def\>{\right>}
\def\({\left(}\def\){\right)}

\def\goto{{\rightarrow}}

\begin{document}

\title{Generalized stepping stone model with $\Xi$-resampling mechanism}
\thanks{Huili Liu's research is supported by NSF of Hebei Province (A2019205299), Hebei Education Department (QN2019073), NSFC (11501164) and HNU (L2019Z01). Xiaowen Zhou's research is supported by Natural Sciences
and Engineering Research Council of Canada (RGPIN-2016-06704) and National Science Foundation of China (11771018).}
\author{Huili Liu and Xiaowen Zhou}
\address{Huili Liu: School of Mathematical Sciences, Hebei Normal University,
Shijiazhuang, Hebei, China}
\email{liuhuili@hebtu.edu.cn}
\address{Xiaowen Zhou: Department of Mathematics and Statistics, Concordia University,
Montreal, Quebec, Canada}
\email{xiaowen.zhou@concordia.ca } 
\date{\today}
\begin{abstract}
A generalized stepping stone model with $\Xi$-resampling mechanism is a two dimensional probability-measure-valued stochastic process whose moment dual is similar to that of the classical stepping stone model except that  Kingman's coalescent is replaced by $\Xi$-coalescent.
We prove the existence of  such a process  by specifying its moments  using the dual function-valued $\Xi$-coalescent process with geographical labels and migration, and then verifying a multidimensional Hausdorff moment problem.
We also characterize the stationary distribution of the generalized stepping stone model and show that it is not reversible if the mutation operator is of uniform jump-type.

\end{abstract}
\maketitle \pagestyle{myheadings} \markboth{\textsc{Generalized stepping stone
model}} {\textsc{Stepping stone model}}

\section{Introduction}
Stepping stone model is a probability-measure-valued stochastic process
describing the evolution of relative frequencies
for different types of alleles in a large
population with geographical structure; see Shiga \cite{Shiga801, Shiga802, Shiga82}, Shiga and Uchiyama \cite{ShigaK}
for earlier work on the classical stepping stone model.
Without geographical structure, namely, the whole population lives in single colony, the model degenerates to the classical Fleming-Viot process  (cf. Fleming and Viot  \cite{FV_79}).
Notohara \cite{Notohara} and Handa \cite{Handa_2} introduced the dual process of the classical stepping stone model as a function-valued Markov process involving Kingman's coalescent respectively.

In the past three decades, more general coalescents have been proposed. For example, the $\Lambda$-coalescent (cf. Pitman \cite{Pitman99} and Sagitov \cite{SS_99}) is a coalescent
with possible multiple collisions and the $\Xi$-coalescent (cf. Sagitov \cite{SS_03} and Schweinsberg \cite{Jason-S}) is a coalescent with possible simultaneous multiple collisions.
It is then interesting to know whether there exists a generalized stepping stone model whose dual is
also a function-valued Markov process evolving in the same way as the dual of the classical stepping stone model but with Kingman's coalescent replaced by a more general
coalescent.
As far as we know, such kind of generalization mainly focus on the single colony model; 
see Birkner et al. \cite{MJM}, Donnelly and Kurtz \cite{DK99b}, Li et al. \cite{LLXZ} and the references therein. However, there has been no breakthrough concerning generalization of the stepping stone model.

In this paper, we will consider replacing Kingman's coalescent in the classical stepping stone model with $\Xi$-coalescent and formulate a more generalized two dimensional probability-measure-valued stochastic process,
named as {\it the generalized stepping stone model with $\Xi$-resampling mechanism}. Intuitively, there is a large population of individuals living in two colonies. Each population undergoes reproduction that is described by a $\Xi$-coalescent with simultaneous multiple collisions, i.e., there are prolific individuals that can simultaneously give birth to children of large amounts comparable to the total population. At the same time, each individual is subject to independent mutation described by an operator $A$ and migration to the other colony at a certain rate.
Instead of considering the well-posedness of martingale problem for superprocess, the moment duality plays an explicit role in establishing the existence of the generalized stepping stone model. This is the most important highlight of our work.
An application of such approach has been introduced in  Evans \cite{Eva97}, where the author proved that subject to a weak duality condition, any system of coalescing Borel right processes gives rise to a Feller semigroup arising from the duality consideration.

Adapting this approach, we show that a function-valued Markov process involving $\Xi$-coalescent,
mutation operator $A$, and mutually geographical migration with rates $u_{\mathbbm{1}}$ and $u_{\mathbbm{2}}$, gives rise to a Feller semigroup, which is sufficient to determine the finite dimensional distributions of  the generalized stepping stone model with $\Xi$-resampling mechanism. Compared with \cite{Eva97}, our coalescing process has more general reproduction mechanism involving simultaneous multiple coalescent.
Further, we prove that
the generalized stepping stone model with $\Xi$-resampling mechanism has a unique invariant measure if the mutation process
allows a unique invariant measure.

The reversibility of a population genetic model is an important issue for statistical inference. Li et al. \cite{LSY99} proved the reversibility of the classical Fleming-Viot process with mutation operator of uniform type. Kermany et al. \cite{AXD} discussed the irreversibility of a two-island diffusion model.
Feng et al. \cite{FZ} proved the irreversibility for an interacting classical Fleming-Viot processes with mutation, selection and recombination.
Li et al. \cite{LLXZ} proved that the $\Xi$-Fleming-Viot process is not reversible except for the degenerate case in \cite{LSY99}.
In this paper we also discuss the reversibility for the generalized stepping stone model with $\Xi$-resampling mechanism. Assuming that the mutation operator is of uniform type, we verify that the process is not reversible.

The plan of the rest of the paper is as follows. In Section \ref{sec2}, we extend the
simultaneous multiple $\Xi$-coalescent to a geographically structured
model with mutation and mutual migration, i.e., the {\it $(\Xi,\text{A},u_{\mathbbm{1}},u_{\mathbbm{2}})$-coalescent process with geographical labels and migration} ({$(\Xi,\text{A},u_{\mathbbm{1}},u_{\mathbbm{2}})$-CPGM}),
which serves as the dual process and is described by a
continuous time function-valued Markov process. In Section \ref{sec3},
the generalized stepping stone model is specified by explicitly defining
a transition semigroup that has its associated ``moments" expressed as expectations for a system of  $(\Xi,\text{A},u_{\mathbbm{1}},u_{\mathbbm{2}})$-CPGMs.
Moreover, we discuss the stationary distribution and irreversibility of the generalized stepping stone model with $\Xi$-resampling mechanism in
Sections \ref{sec5} and \ref{sec6}, respectively.

Throughout the paper, we use $\mathbbm{P}\(\cdot\)$ (resp. $\mathbbm{Q}\(\cdot\)$)
to represent both the probability measure and the associated expectation on probability space $(\Omega,\mathcal{F},\mathbbm{P})$
(resp. $(\Omega^{*},\mathcal{F}^*,\mathbbm{Q})$).

\section{$(\Xi,\text{A},u_{\mathbbm{1}},u_{\mathbbm{2}})$-coalescent process with geographical labels and migration}\label{sec2}
In this section, we first give a short review on $\Xi$-coalescent, and then introduce the $\Xi$-coalescent process with
geographical labels and migration ($(\Xi,u_{\mathbbm{1}},u_{\mathbbm{2}})$-CPGM). Moreover, by adding the mutation operator, we construct a  function-valued Markov process,
named as the $\(\Xi,\text{A},u_{\mathbbm{1}},u_{\mathbbm{2}}\)$-CPGM. Its generator and a coupling property are derived at the end.

\subsection{$\Xi$-coalescent}
Put $[n]=\{1,\ldots,n\}$, $[n]_0=\{0,1,\ldots,n\}$, $[\infty]=\{1,2,\ldots\}$ and $[\infty]_0=\{0,1,2,\ldots\}$.
A {\it partition} of $D\subseteq[\infty]$ is a countable collection
$\pi=\{\pi_i, i=1,2,\ldots\}$ of disjoint blocks such that
$\cup_{i}\pi_i=D$ and $\min\pi_i<\min\pi_j$ for $i<j$.
Denote by $|\cdot|$ the cardinality of the corresponding set. Let
$\mathcal{P}_{D}$ be the collection of partitions for any $D\subseteq[\infty]$. In particular, we write
${\mathbf{1}}_{D}$ as the singleton partition for $D$. For example, ${\mathbf 1}_{[n]}=\{\{1\},\ldots,\{n\}\}$.

Given a partition $\pi\in\mathcal{P}_{D}$ with $|\pi|=n$ and
$\pi'\in\mathcal{P}_{[k]}$ with $n\leq k$, the {\it coagulation} of $\pi$ by $\pi'$, denoted by
$\text{Coag}(\pi,\pi')$, is defined as the following partition of
$D$,
\begin{equation}\label{coag}
\text{Coag}(\pi,\pi')=\pi''=\left\{\pi''_j:=\cup_{i\in \pi'_j}\pi_i, j=1,\ldots,|\pi'|\right\}.
\end{equation}

Given a partition $\pi$ with $|\pi|=n$ and a sequence of positive
integers $s,k_1,\ldots,k_r$ such that $k_i\geq 2, i=1,\ldots,r$
and $n=s+\sum_{i=1}^r k_i $, we say a partition $\pi''$ is obtained
by a {\it $(n;k_1,\ldots,k_r,s)$-collision} of $\pi$ if
$\pi''=\text{Coag}(\pi,\pi')$ for some partition $\pi'$ such that
\[\{|\pi'_i|: i=1,\ldots,|\pi'|
\}=\{k_1,\ldots,k_r,k_{r+1},\ldots,k_{r+s}\},\] where
$k_{r+1}=\cdots=k_{r+s}=1$, i.e., $\pi''$ is obtained by mergering the $n$
blocks of $\pi$ into $r+s$ blocks in which $s$ blocks remain
unchanged and the other $r$ blocks contain $k_1,\ldots,k_r$ blocks
from $\pi$, respectively.

The {$\Xi$-coalescent} is a $\cP_{[\infty]}$-valued Markov process
 $\Pi_{[\infty]}=(\Pi_{[\infty]}(t))_{t\geq 0}$ starting from partition
$\Pi_{[\infty]}{(0)}\in \cP_{[\infty]}$ such that
 for any $D\subseteq[\infty]$, its
restriction to $D$, $\Pi_{D}=(\Pi_{D}(t))_{t\geq 0}$ is a Markov chain
and given that $\Pi_{D}(t)$ has $n$ blocks, each
$(n;k_1,\ldots,k_r;s)$-collision occurs at rate
$\la_{n;\,k_1,\ldots,k_r;\,s}$ with
\begin{equation*}
\begin{split}
\la_{n;\,k_1,\ldots,k_r;\,s}
=\int_\Delta \sum_{\ell=0}^s\sum_{i_1\neq
\cdots\neq i_{r+\ell}}{s\choose{\ell}}x_{i_1}^{k_1}\ldots
x_{i_r}^{k_r}x_{i_{r+1}} \ldots x_{i_{r+\ell}}\left(1-\sum_{j=1}^\infty
x_j\right)^{s-\ell} \frac{\Xi(d\mathbf{x})}{\sum_{j=1}^\infty x_j^2},
\end{split}
\end{equation*}
where $\Xi$ is a finite measure on the infinite simplex
\[\Delta=\left\{\mathbf{x}=(x_1,x_2,\ldots): x_1\geq x_2\geq\cdots\geq 0, \,\sum_{i=1}^\infty x_i\leq 1\right\}.\]
There exists at least one $\pi^{'}\in\mathcal{P}_{[n]}$ that induces the $(n;k_1,\ldots,k_r;s)$-collision. We simply write $$\la_{\pi^{'}}:=\la_{n;\,k_1,\ldots,k_r;\,s}.$$
Given there are $n$ blocks, the total coalescence rate is
$$\la_{n}:=\sum_{\pi^{'}\in\mathcal{P}_{[n]}\setminus\mathbf{1}^{[n]}}\la_{\pi^{'}}.$$

\subsection{$(\Xi,u_{\mathbbm{1}},u_{\mathbbm{2}})$-CPGM}
Let $\mathbbm{S}=\{\mathbbm{1}, \mathbbm{2}\}$ be the labels of two colonies.
$\forall\, n\in [\infty]$, $\mathbbm{S}^n$ is the $n$-fold Cartesian product.
For any $\eta\in \cup_{n=1}^{\infty}\mathbbm{S}^n$ and $\mathbbm{i}\in\{\mathbbm{1,2}\}$,
denote by $|\eta|_{\mathbbm{i}}$
the number of coordinates in $\eta$ equal to $\mathbbm{i}$. 
As a consequence we have $|\eta|=|\eta|_{\mathbbm{1}}+|\eta|_{\mathbbm{2}}$.

Given a partition $\pi=\{\pi_1,\pi_2,\ldots,\pi_n\}$  of $D\subseteq[\infty]$ with cardinality $n$ and a label vector   $\eta=\(\eta_1,\eta_2,\ldots,\eta_n\)\in\mathbbm{S}^n$,
a {\it $\eta$-labeled  partition} $\pi$ is defined as
$$\pi^{\eta}=\left\{\pi_1^{\eta_1},\pi_2^{\eta_2},\ldots,\pi_n^{\eta_n}\right\},~~~~~~~~~~~~~$$
where $\pi_{\ell}^{\eta_{\ell}}$ denotes a {\it $\eta_{\ell}$-labeled block}.
Set $\mathbbm{L}\({\pi}^{\eta}\)=\eta$ and  $\mathbbm{L}\({\pi_i}^{\eta_i}\)=\eta_i$ as the labels; $\mathbbm{B}\({\pi}^{\eta}\)=\pi$ and  $\mathbbm{B}\({\pi_i}^{\eta_i}\)=\pi_i$ as the collection of blocks.
For any $\mathbbm{i}\in\{\mathbbm{1,2}\}$, ${\pi^{\eta}}|_{\mathbbm{i}}$ is the
collection of blocks from $\pi^{\eta}$ with label $\mathbbm{i}$. As usual, the blocks in each collection are always ordered by their least elements.
For any $\pi^{'}\in\mathcal{P}_{[k]}$ with $k\geq |\eta|_{\mathbbm{i}}$, put  $$\text{Coag}^{\mathbbm{i}}\(\pi^{\eta},\pi^{'}\)=\text{Coag}\(\pi^{\eta}|_{\mathbbm{i}},\pi^{'}\)\cup\pi^{\eta}|_{\mathbbm{j}}\,\,\text{with}\,\,\mathbbm{j}
=\{\mathbbm{1,2}\}\setminus\mathbbm{i}$$
where Coag$\(\pi^{\eta}|_{\mathbbm{i}},\pi^{'}\)$ is the coagulation of $\pi^{\eta}|_{\mathbbm{i}}$ by $\pi^{'}$, obtained by (\ref{coag})
but keeping the labels unchanged.

\begin{definition}
The $(\Xi,u_{\mathbbm{1}},u_{\mathbbm{2}})$-CPGM 
is a pure jump and labeled partition-valued Markov process.
Given any initial value,
the $\Xi$-coalescences independently take place among blocks with the same label.
Between coalescence times, with rate $u_{\mathbbm{i}}$, one of the blocks with label $\mathbbm{j}$ is randomly sampled with its label  replaced by label $\mathbbm{i}$, where $\{\mathbbm{i,\,j}\}=\{\mathbbm{1,\,2}\}$.
\end{definition}

\subsection{$\(\Xi,\text{A},u_{\mathbbm{1}},u_{\mathbbm{2}}\)\text{-CPGM}$}
Let $E=[0,1]$ be the state space. $\mathcal{E}=\mathscr{B}\(E\)$ is the Borel $\sigma$-algebra on $E$.
Set
\begin{equation*}
\begin{split}
&\mathcal{D}_0=\left\{E,\emptyset\right\};\\
&\mathcal{D}_n=\left\{\left[{i}/{2^n},{\(i+1\)}/{2^n}\right),\,i=0,1,\ldots,2^n-2,\right\}\cup{\big\{\left[{\(2^n-1\)}/{2^n},1\right]\big\} }\,\,\text{for}\,\,n\geq1.
\end{split}
\end{equation*}
Define $\mathcal{D}$ as the ring generated by
$\cup_{n=0}^{\infty}D_n$. Clearly, $\mathcal{D}$ is countable.
Let $\sigma\(\mathcal{D}\)$ be the $\sigma$-algebra generated by $\mathcal{D}$. Subsequently, $\mathcal{E}=\sigma\(\mathcal{D}\)$. Write
$$\mathcal{D}\times\mathcal{D}=\left\{C_i\times D_j |C_i,\,D_j\in\mathcal{D},\,i,j=1,2,\ldots\right\}.$$
One can see that $\sigma\(\mathcal{D}\times\mathcal{D}\)=\sigma\(\mathcal{E}\times\mathcal{E}\)=\mathscr{B}\(E^2\)$.
Let ${B}(E)$ be the real-valued bounded 
functions on $E$ and ${C}(E)$ be the real-valued continuous functions on $E$.
Denote by ${M}_1(E)$ the collection of probability measures on $E$.
${B}(E^n)$, ${C}(E^n)$ and ${M}_1(E^n)$ are defined similarly on $E^n$.

The mutation operator $A$ is a jump-type Feller generator such that
\begin{equation}\label{eq:mu}
Af(x)=\frac{\theta}{2}\int_{E}\(f(y)-f(x)\)q(x,dy),
\end{equation}
where $f(x)\in B(E)$, $\theta>0$ is a constant,  and $q(x,\Gamma)$ is a probability transition function on $E\times \mathcal{E}$. Let $\(T_t\)_{t\geq0}$ be the Feller semigroup associated to $A$.
Denote by $A^{(n)}$ the linear operator on
$B\(E^n\)$ that generates the Feller semigroup $\(T^{(n)}_t\)_{t\geq
0}$ corresponding to $n$ independent copies of the processes
associated  to $(T_t)_{t\geq 0}$. Then
\begin{equation*}
\begin{split}
A^{\(n\)}f(x_1,\ldots,x_n)=\,&\sum_{k=1}^n\frac{\theta}{2}\int_{E}\left[f\(x_1,\ldots,x_{k-1},y,x_{k+1},\ldots,x_{n}\)\right.\\
&\quad\quad\quad\quad\left.-f\(x_1,\ldots,x_{k-1},x_k,x_{k+1},\ldots,x_{n}\)\right]q\(x_k,dy\).
\end{split}
\end{equation*}

Given any ${\eta}=\(\eta_1,\ldots,\eta_n\)\in
{\mathbbm{S}}^{n}$, $\mu=\(\mu_{\mathbbm{1}},\mu_{\mathbbm{2}}\)\in{M}_1(E^2)$ and $f\in B\(E^n\)$, we denote by
\begin{equation}\label{G_{f,u}}
\begin{split}
\mathcal{G}_{f,\,\eta}\(\mu\)&\,:=\,\mathcal{G}_{\mu}\(f,\eta\):=\,\left<\mu_{\eta},f\right>\\
&\,:=\,\int_{E}
\cdots\int_{E}f\(x_1,\ldots,x_n\)\mu_{{\eta}_1}\(dx_1\)\cdots\mu_{{\eta}_n}\(dx_n\).
\end{split}
\end{equation}
In particular, for $f\(x_1,\ldots,x_n\)=\prod_{k=1}^nf_k(x_k)$ with $f_k\in{B}\(E\)$, $k=1,\ldots,n$, we have
$$A^{\(n\)}f(x_1,\ldots,x_n)=\sum_{k=1}^n A f_k\(x_k\)\prod_{\ell\neq k}f_{\ell}\(x_{\ell}\)$$
and
\begin{equation*}\label{seq-fun}
\begin{split}
\mathcal{G}_{\mu}\(f,\eta\)
\,=&\,\prod_{\{k|\eta_k=\mathbbm{1},\,k=1,\ldots,n\}}\<\mu_{\mathbbm{1}},f_k\>\prod_{\{k|\eta_k=\mathbbm{2},\,k=1,\ldots,n\}}\<\mu_{\mathbbm{2}},f_k\>\\
:=&\,\<\mu_{\mathbbm{1}}^{\otimes |\eta|_{\mathbbm{1}}},f^{\mathbbm{1}}\>\<\mu_{\mathbbm{2}}^{\otimes |\eta|_{\mathbbm{2}}},f^{\mathbbm{2}}\>
\end{split}
\end{equation*}
with
\[f^{\mathbbm{i}}=\otimes_{\{k|\eta_k=\mathbbm{i},\,k=1,\ldots,n\}}f_k\(x_k\)\,\,\text{for}\,\,\mathbbm{i}\in\{\mathbbm{1,2}\}.\]

We start to construct a $\cup_{m=1}^{\infty}{B}\(E^m\)\times \mathbbm{S}^m$-valued Markov process $\(Y\(t\),\eta\(t\)\)_{t\geq 0}$ on some probability space $(\Omega,\mathcal{F},\mathbbm{P})$.
Given initial value
$$\(Y(0),\eta\(0\)\)\in B\(E^n\)\times\mathbbm{S}^n,$$
we associate the process with a  $(\Xi,u_{\mathbbm{1}},u_{\mathbbm{2}})$-CPGM whose initial value is denoted by a singleton partition with label $\eta\(0\)$, i.e.,
${\mathbf{1}}_{[n]}^{\eta\(0\)}=\left\{\{1\}^{\eta\(0\)_1},\ldots,\{n\}^{{\eta}\(0\)_{n}}\right\}.$
$\(\eta\(t\)\)_{t>0}$ is defined as the label-valued process of the $(\Xi,u_{\mathbbm{1}},u_{\mathbbm{2}})$-CPGM. It is obvious that the value of $\eta\(t\)$ does not change until either coalescence or migration happens. Let$\{\tau_m\}_{m\geq1}$ be the sequence of jumping times of the associated $(\Xi,u_{\mathbbm{1}},u_{\mathbbm{2}})$-CPGM. Set $\tau_0=0$. $\(Y\(t\)\)_{t\geq 0}$ is formulated recursively. $\forall\,m\in[\infty]_0$,  define
$${Y}(t)=T^{\(|\eta(\tau_m)|\)}_{t-\tau_m}Y(\tau_m)\text{~~for any~~} t\in\left[\tau_m,
\tau_{m+1}\right).$$
At jumping time $\tau_{m+1}$, we need to introduce the operator concerning with the reduction of number of variables if coalescence happens.
For any $\mathbbm{i}\in\{\mathbbm{1,2}\}$ and $\pi\in\mathcal{P}_{[|\eta\(\tau_{m+1-}\)|_{\mathbbm{i}}]}\setminus{\mathbf 1}_{\left[|\eta\(\tau_{m+1-}\)|_{\mathbbm{i}}\right]}$, denote by
$$\pi\(\mathbbm{i}\)=\mathbbm{B}\(\text{Coag}^{\mathbbm{i}}\(\mathbf{1}_{[|\eta\(\tau_{m+1-}\)|]}^{\eta\(\tau_{m+1-}\)},\pi\)\)=\{\pi_{\ell},\ell=1,2,\ldots,|\pi\(\mathbbm{i}\)|\}.$$
A map $\Phi^{\mathbbm{i}}_{\pi}$ from ${B}\(E^{|\eta\(\tau_{m+1-}\)|}\)$ to ${B}\(E^{|\pi(\mathbbm{i})|}\)$ is defined as
$$\Phi^{\mathbbm{i}}_{\pi}g\(x_{1},x_2,\ldots,x_{|\eta\(\tau_{m+1-}\)|}\)=g\(x_{{i}_1},x_{{i}_2},\ldots,x_{{i}_{|\eta\(\tau_{m+1-}\)|}}\)$$
with $i_j=\ell$ for $i_j\in\pi_{\ell}$, $j\in\left[|\eta\(\tau_{m+1-}\)|\right]$.
Then we have
\begin{equation*}
\begin{split}
&\mathbb{P}\(Y\(\tau_{m+1}\)
=\Phi^\mathbbm{i}_{\pi}Y\(\tau_{m+1}-\)|\(\eta\(t\)\)_{0\leq t\leq \tau_{m+1}-}\)\\
&\quad=
\ind_{\left\{|\eta(\tau_{m+1})|_{\mathbbm{i}}<|\eta(\tau_{m+1}-)|_{\mathbbm{i}},\,
\pi\in {\mathcal{P}_{\left[\left|\eta\(\tau_{m+1}-\)\right|_{\mathbbm{i}}\right]}\setminus
\mathbf{1}_{\left[\left|\eta\(\tau_{m+1}-\)\right|_{\mathbbm{i}}\right]}}\right\}}\,\,{\text{with}}\,\,\mathbbm{i}\in\{\mathbbm{1,2}\};\\
&\mathbb{P}\(Y\(\tau_{m+1}\)=Y\(\tau_{m+1}-\)|\(\eta\(t\)\)_{0\leq t\leq \tau_{m+1}-}\)=\ind_{\left\{|\eta(\tau_{m+1})|=|\eta(\tau_{m+1}-)|\right\}},
\end{split}
\end{equation*}
where $\ind_{\omega}$ denotes the indicator function of an
event $\omega$. The first equation is due to the coalescence within colony $\mathbbm{i}$; the second one is due to the migration between the two colonies.

The $\cup_{m=1}^{\infty}{B}\(E^m\)\times {\mathbbm{S}}^m$-valued Markov process $\(Y\(t\),\eta\(t\)\)_{t\geq 0}$ constructed above is called the {\it  $\(\Xi,\text{A},u_{\mathbbm{1}},u_{\mathbbm{2}}\)\text{-CPGM}$}.

\begin{proposition}\label{ge-pro}
Given $\(Y(0),\eta(0)\)=\(f,\eta\)\in\cup_{m=1}^{\infty}{B}\(E^m\)\times {\mathbbm{S}}^m$, $\forall$ $\mu\in M_1\(E^2\)$, the generator $\mathcal{L}$ of the $\(\Xi,\text{A},u_{\mathbbm{1}},u_{\mathbbm{2}}\)$-CPGM $\(Y\(t\),\eta\(t\)\)_{t\geq0}$ is given by
{\rm
\begin{equation}\label{generator-dual}
\begin{split}
&\mathcal{L}\mathcal{G}_{\mu}\(f,\eta\)\\
&=\mathcal{G}_{\mu}\(A^{\(|\eta|\)}f,\eta\)+\sum_{\mathbbm{i}\in\{\mathbbm{1,2}\}}\sum_{\pi\in\mathcal{P}_{\left[|\eta|_{\mathbbm{i}}\right]}\setminus\mathbf{1}^{\left[|\eta|_{\mathbbm{i}}\right]}}
\la_{\pi}\(\mathcal{G}_{\mu}\(\Phi_{\pi}^{\mathbbm{i}}f,{\beta}_{\pi}^{\mathbbm{i}}\(\eta\)\)-\mathcal{G}_{\mu}\(f,\eta\)\)\\
&\qquad+\sum_{\{\mathbbm{i,\,j}\}=\{\mathbbm{1,2}\}}\sum_{\{k|\eta_k={\mathbbm{j}},\,k=1,\ldots,|\eta|\}}u_{\mathbbm{i}}
\(\mathcal{G}_{\mu}\(f,{\gamma}_{k,\,\mathbbm{i}}\(\eta\)\)-\mathcal{G}_{\mu}\(f,\eta\)\),
\end{split}
\end{equation}
}
where
{\rm
\begin{eqnarray*}
\begin{split}
&\be^{\mathbbm{i}}_{\pi}\(\eta\)=\mathbbm{L}\({\text{Coag}}^{\mathbbm{i}}\({\mathbf 1}_{[|\eta|]}^{\eta},\pi\)\)\text{\,for any\,}\pi\in \mathcal{P}_{\left[|\eta|_{\mathbbm{i}}\right]}\setminus{\mathbf 1}_{\left[|\eta|_{\mathbbm{i}}\right]};\\
&\ga_{k,\,\mathbbm{i}}\(\eta\)=\(\eta_1,\ldots,\eta_{k-1},\mathbbm{i},\eta_{k+1},\ldots,\eta_{|\eta|}\),\,\, 1\leq k\leq |\eta|.
\end{split}
\end{eqnarray*}
}

\end{proposition}

\begin{proof}
The result is easily obtained from the construction of the $\(\Xi,\text{A},u_{\mathbbm{1}},u_{\mathbbm{2}}\)$-CPGM.
\end{proof}

Next we present a coupling property for the $(\Xi,\text{A}, u_{\mathbbm{1}},u_{\mathbbm{2}})$-CPGM. It implies that the impact of a  $\(\Xi,\text{A},u_{\mathbbm{1}},u_{\mathbbm{2}}\)$-CPGM is the same as that of superposition of countable infinitely many  $\(\Xi,\text{A},u_{\mathbbm{1}},u_{\mathbbm{2}}\)$-CPGMs as long as they are all driven by a  $\(\Xi,u_{\mathbbm{1}},u_{\mathbbm{2}}\)$-CPGM
and the superposition of initial values coincides with that of  the single process.

\begin{proposition}\label{coupling property}
Let $\(Y_i\(t\),\eta\(t\)\)_{t\geq 0},\, i\in[\infty]$ be a sequence of $\(\Xi,\text{A},u_{\mathbbm{1}},u_{\mathbbm{2}}\)$-CPGMs that are driven by a  $(\Xi,u_{\mathbbm{1}},u_{\mathbbm{2}})$-CPGM.
Given any $\mu\in{M}_1(E^2)$, $\forall$ $t\geq0$, if
\begin{equation*}
\sum_{i=1}^{\infty}\mid\mathcal{G}_{\mu}\(Y_i\(t\),\eta\(t\)\)\mid<\infty, 
\end{equation*}
one can see that
\begin{equation*}
\sum_{i=1}^{\infty}\mathcal{G}_{\mu}\(Y_i\(t\),\eta\(t\)\)=\mathcal{G}_{\mu}\(Y\(t\),\eta\(t\)\),
\end{equation*}
where $\(Y\(t\),\eta\(t\)\)$ is also a $\(\Xi,\text{A},u_{\mathbbm{1}},u_{\mathbbm{2}}\)$-CPGM driven by the same $(\Xi,u_{\mathbbm{1}},u_{\mathbbm{2}})$-CPGM with initial value $Y\(0\)=\sum_{i=1}^{\infty}Y_i\(0\)$. 
\end{proposition}

\begin{proof}
For any $s\in[0,\tau_1)$, one can see that
\begin{equation*}
\sum_{i=1}^{\infty}Y_i\(s\)=\sum_{i=1}^{\infty}T_s^{\otimes |\eta\(0\)|}Y_i\(0\)=T_s^{\otimes |\eta\(0\)|}Y\(0\)=Y\(s\).
\end{equation*}
Further, if $\tau_1$ is a migration time, then
\begin{equation*}
\sum_{i=1}^{\infty}Y_i\(\tau_1\)=\sum_{i=1}^{\infty}Y_i\(\tau_1-\)=Y\(\tau_1-\)=Y\(\tau_1\);
\end{equation*}
otherwise, $\tau_1$ is a coalescence time. Since the deduction of variables in $Y_i\(\tau_1-\)$, $i\in[\infty]$ and $Y\(\tau_1-\)$ is driven by the same coalescence event. Subsequently, for any $\Phi^\mathbbm{i}_{\pi}$ with $\mathbbm{i}\in\{\mathbbm{1,2}\}$ and $\pi\in  \mathcal{P}_{\left[|\eta\(\tau_1-\)|_{\mathbbm{i}}\right]}\setminus{\mathbf 1}_{\left[|\eta\(\tau_1-\)|_{\mathbbm{i}}\right]}$, we still have
\begin{equation*}
\sum_{i=1}^{\infty}Y_i\(\tau_1\)=\sum_{i=1}^{\infty}\Phi^{\mathbbm{i}}_{\pi}Y_i\(\tau_1-\)
=\Phi^{\mathbbm{i}}_{\pi}\(\sum_{i=1}^{\infty}Y_i\(\tau_1-\)\)=\Phi^{\mathbbm{i}}_{\pi}Y\(\tau_1-\)=Y\(\tau_1\).
\end{equation*}
By mathematical induction on the jumping times, one can see that $\sum_{i=1}^{\infty}Y_i\(t\)=Y\(t\)$ holds for any $t\geq 0$.
Then the desired result follows from (\ref{G_{f,u}}).

\end{proof}

\section{Generalized stepping stone model with $\Xi$-resampling mechanism}\label{sec3}
In this section, we formulate {\it the generalized stepping stone model with $\Xi$-resampling mechanism} to describe the gene frequencies for a population model with two colonies.
We start with introducing the multidimensional Hausdorff moment problem.
Then, subject to moment duality, the transition semigroup of the generalized stepping stone model with $\Xi$-resampling mechanism is uniquely defined from its dual $\(\Xi,\text{A},u_{\mathbbm{1}},u_{\mathbbm{2}}\)$-CPGM. As a consequence, the transition semigroup with initial value determines the finite dimensional distributions of the generalized stepping stone model.

\subsection{Multidimensional Hausdorff moment problem}
The lemma below provides some necessary and sufficient conditions for the solution to multidimensional Hausdorff moment problem.

\begin{lemma}[cf. Proposition 4.6.11 in Berg et al. \cite{Berg}]\label{Haus}
Given any $k\in[\infty]$,  for a function $\psi:[\infty]_0^k\rightarrow\mathbb{R}$, the following conditions are equivalent:
\begin{enumerate}[(i)]
\item $\psi$ is completely monotone;
\item
$$\sum_{\mathbf{0}\leq \mathbf{p}\leq \mathbf{n}}\(-1\)^{\parallel\mathbf{p}\parallel}{\mathbf{n}\choose\mathbf{p}}
\psi\(\mathbf{m}+\mathbf{p}\)\geq0\text{~~~~~~for all~~}\mathbf{n},\mathbf{m}\in[\infty]_0^k,$$
where {$\bf{\leq}$} denotes the usual coordinatewise partial order on $[\infty]_0^{k}$,
$\mathbf{n}=\(n_1,\ldots,n_k\)$, $\mathbf{m}=\(m_1,\ldots,m_{k}\)$, $\mathbf{p}=\(p_1,\ldots,p_k\)$,
$\parallel\mathbf{p}\parallel=p_1+\cdots+p_k$  and
\[{\mathbf{n}\choose\mathbf{p}}=\prod_{i=1}^{k}{{n_i\choose p_i}};\]
\item There exists $\mu\in{M}_{+}\([0,1]^k\)$ such that
$$\psi(\mathbf{n})=\int_{[0,1]^k}\mathbf{x}^{\mathbf{n}}d\mu\({\mathbf{x}}\),~~~~{\mathbf{n}\in[\infty]_0^{k}},$$
where ${M}_{+}\(X\)$ is the collection of Radon measures on $X$.
\end{enumerate}
\end{lemma}

\subsection{Generalized stepping stone model with $\Xi$-resampling mechanism}
Denote by
$\(C\(M_1\(E^2\)\),\|\cdot\|\)$ the Banach space of continuous functions on $M_1\(E^2\)$
with norm $\|\mathbbm{F}\|=\sup_{\mu\in M_1\(E^2\)}|\mathbbm{F}\(\mu\)|$. 
In order to formulate a measure-valued stochastic process, we need to introduce some subspace of $C\(M_1\(E^2\)\)$ serving as the core for the generator.
Denote by $C_p\(M_1\(E^2\)\)$  the linear span of monomials of the form
\begin{equation*}
\mathbb{F}_{f,\,\eta,\,n}\(\mu\)=\int_{E}\cdots\int_{E}f\(x_1,\ldots,x_n\)
\prod_{i=1}^n\mu_{\eta_i}\(dx_i\)
\end{equation*}
with $n\in[\infty]$, $f\(x_1,\ldots,x_n\)=f_1\(x_1\)\cdots f_n\(x_n\)\in C\(E\)^{n}$
 and $\eta=\(\eta_1,\ldots,\eta_n\)\in\mathbbm{S}^n$. It follows from the Stone-Weierstrass Theorem 
that ${C}_p({M}_1\(E^2\))$ is dense in ${C}({M}_1\(E^2\))$.
Let $B_{b,\,p}({M}_1\(E^2\))$ be the collection of those $\mathbb{F}_{f,\,\eta,\,n}\(\mu\)$ with $n\in[\infty]$, $f\in{B}\(E^n\)$
and $\eta\in\mathbbm{S}^n$. Denote by $\bar{B}_{b,\,p}\({M}_1\(E^2\)\)$ the closure of ${B}_{b,\,p}\({M}_1\(E^2\)\)$. $B\(M_1\(E^2\)\)$ is the collection of bounded functions on $M_1\(E^2\)$. Clearly, we have ${C}_p({M}_1\(E^2\))\subseteq B_{b,\,p}({M}_1\(E^2\))\subseteq\bar{B}_{b,\,p}\({M}_1\(E^2\)\)\subseteq B\(M_1\(E^2\)\)$.
Throughtout this paper, we make the following assumptions:
\begin{enumerate}[(I)]
\item A finite measure $\Xi$ on 
$\Delta$ characterizes the simultaneous multiple coalescent;
\item The mutation operator $A$ is a jump-type Feller generator given by (\ref{eq:mu});
\item The migration rates $u_{\mathbbm 1}$ and $u_{\mathbbm 2}$ are positive.
\end{enumerate}

\begin{proposition}\label{cor1}
For each $\mu=\(\mu_{\mathbbm{1}},\mu_{\mathbbm{2}}\)\in M_1\(E^2\)$ and $t\geq0$,
there exists a $M_1\(E^2\)$-valued random variable $V_{\mu,t}$ on some complete probability space $\(\Omega^{*},\mathcal{F}^{*},\mathbbm{Q}\)$
such that for all $k$, $\ell\in[\infty]$ with $k+\ell>0$,
$\mathbf{n_1}=(n_{11}, \ldots, n_{1k})\in\left[\infty\right]_0^k$, $\mathbf{n_2}=(n_{21}, \ldots, n_{2\ell})\in\left[\infty\right]_0^{\ell}$,
$\{C_i\times D_j\}_{i\in [k],\,j\in[\ell]}\subseteq\mathcal{D}\times\mathcal{D}$, we have
{\rm
\begin{equation*}
\begin{split}
&\mathbbm{Q}\left[V_{\mu,\,t}\(\otimes_{i=1}^k\ind_{C_i}^{\otimes n_{1i}}\otimes_{i=1}^{\ell}\ind_{D_i}^{\otimes n_{2i}}\)\right]
=\mathbbm{P}\left[\mathcal{G}_{\mu}\(Y_{\mathbf{n_1,\,n_2}}\(t\),\eta_{\mathbf{n_1,\,n_2}}\(t\)\)\right],
\end{split}
\end{equation*}
}
where $\(Y_{\mathbf{n_1,\,n_2}}\(t\),\eta_{\mathbf{n_1,\,n_2}}\(t\)\)_{t\geq0}$ is a $\(\Xi,\text{A},u_{\mathbbm{1}},u_{\mathbbm{2}}\)$-CPGM with initial value
{\rm$$ Y_{\mathbf{n_1,\,n_2}}\(0\)=\otimes_{i=1}^k\ind_{C_i}^{\otimes n_{1i}}\otimes_{i=1}^{\ell}\ind_{D_i}^{\otimes n_{2i}}\,\,\,\,\text{and}\,\,\,\,\eta_{\mathbf{n_1,\,n_2}}\(0\)=\{\underbrace{\mathbbm{1},\ldots,\mathbbm{1}}_{n_1},\underbrace{\mathbbm{2},\ldots,\mathbbm{2}}_{n_2}\}.$$
}

\end{proposition}

\begin{proof}
Given that $Y_{\mathbf{n_1,\,n_2}}\(0\)$ is in the form of tensor product of one variable functions, one can see that $Y_{\mathbf{n_1,\,n_2}}\(t\)$ is always in a similar form of one variable functions from the construction of the $\(\Xi,\text{A},u_{\mathbbm{1}},u_{\mathbbm{2}}\)$-CPGM.  Subsequently, given any $t\geq 0$, we may  assume that
\begin{equation*}
\begin{split}
Y_{\mathbf{n_1,\,n_2}}\(t\)&=g_1\(x_1\)g_2\(x_2\)\ldots g_{|\eta_{\mathbf{n_1,\,n_2}}\(t\)|}\(x_{|\eta_{\mathbf{n_1,\,n_2}}\(t\)|}\)\\
:&= Y_{\mathbf{n_1}}^{\mathbbm{1}}\(t\)\otimes Y_{\mathbf{n_2}}^{\mathbbm{2}}\(t\)
\end{split}
\end{equation*}
with
$$Y_{\mathbf{n_1}}^{\mathbbm{1}}\(t\)=\otimes_{\{i|\eta_{\mathbf{n_1,\,n_2}}\(t\)_i=\mathbbm{1},\,i=1,\ldots,|\eta_{\mathbf{n_1,\,n_2}}\(t\)|\}}g_i\(x_i\);$$
$$Y_{\mathbf{n_2}}^{\mathbbm{2}}\(t\)=\otimes_{\{i|\eta_{\mathbf{n_1,\,n_2}}\(t\)_i=\mathbbm{2},\,i=1,\ldots,|\eta_{\mathbf{n_1,\,n_2}}\(t\)|\}}g_i\(x_i\).$$
For simplicity of notation, let's write
$$\al_{\mathbf{n_1}}^{\mathbbm{1}}(t)=|\eta_{\mathbf{n_1,\,n_2}}(t)|_{\mathbbm{1}};\,\, \al_{\mathbf{n_2}}^{\mathbbm{2}}(t)=|\eta_{\mathbf{n_1,\,n_2}}(t)|_{\mathbbm{2}}$$
and set $\|\mathbf{v}\|=v_{1}+\cdots+v_{k}$ for any $\mathbf{v}=\(v_{1},\ldots,v_{k}\)\in[\infty]_0^k$.
Define
\begin{eqnarray}\label{F_{n,m}}
\begin{split}
F(\mathbf{n_1},\,\mathbf{n_2}):&=\mathbbm{P}\left[\mathcal{G}_{\mu}\(Y_{\mathbf{n_1,\,n_2}}(t),\,\eta_{\mathbf{n_1,\,n_2}}(t)\)\right]\\
&=\mathbbm{P}\left[\<\mu_{\mathbbm{1}}^{\otimes\al_{\mathbf{n_1}}^{\mathbbm{1}}(t)},
Y_{\mathbf{n_1}}^{{\mathbbm{1}}}(t)\>\<{\mu_{\mathbbm{2}}}^{\otimes\al_{\mathbf{n_2}}^{\mathbbm{2}}(t)},
Y_{\mathbf{n_2}}^{\mathbbm{2}}(t)\>\right].
\end{split}
\end{eqnarray}

By Lemma \ref{Haus}, in order to show that (\ref{F_{n,m}}) is the joint moment of some measure, it suffices to verify that
\begin{eqnarray*}
\begin{split}
&I:\,=\sum_{\mathbf{0}\leq\mathbf{p}\leq\mathbf{r}}\sum_{\mathbf{0}\leq\mathbf{q}\leq
\mathbf{h}}(-1)^{\|\mathbf{p}\|+\|\mathbf{q}\|}  {\mathbf{r}\choose\mathbf{p}}{\mathbf{h}\choose\mathbf{q}}
F(\mathbf{n_1}+\mathbf{p},\,
\mathbf{n_2}+\mathbf{q})\geq 0
\end{split}
\end{eqnarray*}
for any $\mathbf{r}=(r_1, \ldots, r_k)\in[\infty]_0^k$ and $\mathbf{h}=\(h_1,\ldots,h_{\ell}\)\in[\infty]_0^{\ell}$,
$\mathbf{p}=(p_1, \ldots, p_k)$ with $p_i\leq r_i$ and
$\mathbf{q}=\(q_1,\ldots,q_{\ell}\)$ with $q_i\leq h_i$.
It follows from (\ref{F_{n,m}}) that
\begin{eqnarray*}
\begin{split}
I=\sum_{\mathbf{0}\leq\mathbf{p}\leq\mathbf{r}}\sum_{\mathbf{0}\leq\mathbf{q}\leq\mathbf{h}}{\mathbf{r}\choose\mathbf{p}}{\mathbf{h}\choose\mathbf{q}}
(-1)^{\|\mathbf{p}\|+\|\mathbf{q}\|}\mathbbm{P}\left[\<\mu_{\mathbbm{1}}^{\otimes\al^{\mathbbm{1}}_{\mathbf{n_1}+\mathbf{p}}(t)},
Y_{\mathbf{n_1}+\mathbf{p}}^{\mathbbm{1}}(t)\>\<\mu_{\mathbbm{2}}^{\otimes\al^{\mathbbm{2}}_{\mathbf{n_2}+\mathbf{q}}(t)},
Y_{\mathbf{n_2}+\mathbf{q}}^{\mathbbm{2}}(t)\>\right]
\end{split}
\end{eqnarray*}
with initial values
$$Y_{\mathbf{n_1}+\mathbf{p}}^{\mathbbm{1}}(0)=\otimes_{i=1}^k\ind_{C_i}^{\otimes
n_{1i}+p_i}; \, Y_{\mathbf{n_2}+\mathbf{q}}^{\mathbbm{2}}(0)=\otimes_{i=1}^{\ell}\ind_{D_i}^{\otimes
n_{2i}+q_i};$$
$$\al^{\mathbbm{1}}_{\mathbf{n_1}+\mathbf{p}}(0)=\|\mathbf{n_1}\|+\|\mathbf{p}\|; \,
\al^{\mathbbm{2}}_{\mathbf{n_2}+\mathbf{q}}(0)=\|\mathbf{n_2}\|+\|\mathbf{q}\|.$$
For $\mathbf{r}$ and $\mathbf{h}$, let's write
\[\mathscr{C}_\mathbf{r}:=\{c=c_1\times\cdots\times c_k: c_i\subset [r_i]_0,\,
1\leq i\leq k\};\]
\[{\mathscr{C}}_\mathbf{h}:=\{\tilde{c}=\tilde{c}_1\times\cdots\times \tilde{c}_{\ell}: \tilde{c}_i\subset [h_i]_0,\,
1\leq i\leq {\ell}\}.\]
Put $\mathbf{c}=(|c_1|,\ldots,|c_k|)\in[\infty]_0^k$ 
and $\tilde{\mathbf{c}}=(|\tilde c_1|,\ldots,|\tilde
c_{\ell}|)\in[\infty]_0^{\ell}$. 
Then
\begin{eqnarray*}
\begin{split}
I&\,=\sum_{\scriptstyle\mathbf{0}\leq\mathbf{p}\leq\mathbf{r}}\sum_{\scriptstyle\mathbf{0}\leq\mathbf{q}\leq\mathbf{h}}\sum_{\scriptstyle c\in
{\mathscr{C}}_\mathbf{r},\atop \scriptstyle |c_i|=p_i,\,i\in[k]} \sum_{\scriptstyle \tilde c\in
{\mathscr{C}}_\mathbf{h},\atop \scriptstyle |\tilde c_i|=q_i,\,
i\in[\ell]}(-1)^{\|\mathbf{p}\|+\|\mathbf{q}\|}\\
&\qquad\times\mathbbm{P}\left[\<\mu_{\mathbbm{1}}^{\otimes\al^{\mathbbm{1}}_{\mathbf{n_1+c}}(t)},
Y_{\mathbf{n_1+c}}^{\mathbbm{1}}(t)\>\<\mu_{\mathbbm{2}}^{\otimes\al^{\mathbbm{2}}_{\mathbf{n_2}+\tilde{\mathbf{c}}}(t)},
Y_{\mathbf{n_2}+\tilde{\mathbf{c}}}^{\mathbbm{2}}(t)\>\right]\\
&\,=\sum_{c\in {\mathscr{C}}_{\mathbf{r}}}\sum_{\tilde c\in {\mathscr{C}}_{\mathbf{h}}}(-1)^{\|\mathbf
c\|+\|\tilde{\mathbf c}\|}\mathbbm{P}\left[\<\mu_{\mathbbm{1}}^{\otimes\al^{\mathbbm{1}}_{\mathbf{n_1+c}}(t)},
Y_{\mathbf{n_1+c}}^{\mathbbm{1}}(t)\>\<\mu_{\mathbbm{2}}^{\otimes\al^{\mathbbm{2}}_{\mathbf{n_2}+\tilde{\mathbf{c}}}(t)},
Y_{\mathbf{n_2}+\tilde{\mathbf{c}}}^{\mathbbm{2}}(t)\>\right]\\
\end{split}
\end{eqnarray*}
with initial values
$$Y_{\mathbf{n_1+c}}^{\mathbbm{1}}(0)=\otimes_{i=1}^k\(\otimes_{j=1}^{n_{1i}}\ind_{C_i}\otimes_{j=n_{1i}+1}^{n_{1i}+r_i}h_{ij}\);\,\,
Y_{\mathbf{n_2}+\tilde{\mathbf
c}}^{\mathbbm{2}}(0)=\otimes_{i=1}^{\ell}\(\otimes_{j=1}^{n_{2i}}\ind_{D_i}\otimes_{j=n_{2i}+1}^{n_{2i}+h_i}g_{ij}\);$$
$$\al^{\mathbbm{1}}_{\mathbf{n_1+c}}(0)=\|\mathbf{n_1}\|+\|\mathbf c\|;\,\,
\al^{\mathbbm{2}}_{\mathbf{n_2}+\tilde{\mathbf{c}}}(0)=\|\mathbf{n_2}\|+\|\tilde{\mathbf c}\|;$$
where
\begin{equation*}
\begin{aligned}
h_{ij}=\left\{\begin{array}{lll}
\ind_{C_i}\,\,&\mbox{if\,\,} j-n_{1i}\in c_i;\\
\ind      &\text{if\,\,} j-n_{1i}\not\in c_i;
\end{array}
\right.\\
\end{aligned}
\begin{aligned}
\quad\quad
g_{ij}=\left\{\begin{array}{lll}
\ind_{D_i} &\mbox{if\,\,} j-n_{2i}\in \tilde c_i;\\
       \ind       &\text{if\,\,} j-n_{2i}\not\in \tilde c_i.\\
\end{array}
\right.
\end{aligned}
\end{equation*}
Therefore,
\begin{eqnarray*}
\begin{split}
I&\,=\mathbbm{P}\left[\<\mu_{\mathbbm{1}}^{\otimes\al^{\mathbbm{1}}_{\mathbf {n_1}+\mathbf r}(t)}, Y_{\mathbf
{n_1}+\mathbf r}^{\mathbbm{1}}(t)\>\<\mu_{\mathbbm{2}}^{\otimes\al^{\mathbbm{2}}_{\mathbf{n_2}+\mathbf h}(t)},
Y_{\mathbf{n_2}+\mathbf h}^{\mathbbm{2}}(t)\>\right]\\
&\,=\mathbbm{P}\left[\mathcal{G}_{\mu}\(Y_{\mathbf{n_1}+\mathbf r,\,\mathbf{n_2}+\mathbf h}(t), \eta_{\mathbf{n_1}+\mathbf r,\,\mathbf{n_2}+\mathbf
h}(t)\)\right]
\,\geq 0
\end{split}
\end{eqnarray*}
with initial values
$$Y_{\mathbf{n_1}+\mathbf
r}^{\mathbbm{1}}(0)=\otimes_{i=1}^k\left[\otimes_{j=1}^{n_{1i}}\ind_{C_i}\otimes_{j=n_{1i}+1}^{n_{1i}+r_i}(\ind-\ind_{C_i})\right];$$
$$Y_{\mathbf{n_2}+\mathbf
h}^{\mathbbm{2}}(0)=\otimes_{i=1}^{\ell}\left[\otimes_{j=1}^{n_{2i}}\ind_{D_i}\otimes_{j=n_{2i}+1}^{n_{2i}+h_i}(\ind-\ind_{D_i})\right];$$
$$\al^{\mathbbm{1}}_{\mathbf{n_1}+\mathbf r}(0)= \|\mathbf{n_1}\|+\|\mathbf r\|;\,\,\al^{\mathbbm{2}}_{\mathbf{n_2}+\mathbf h}(0)= \|\mathbf{n_2}\|+\|\mathbf h\|.$$

Consequently, for any $\mathbf{n_1}\in[\infty]_0^k$ and
$\mathbf{n_2}\in[\infty]_0^{\ell}$ with $k+\ell>0$, there exist $[0, 1]$-valued random
variables $W_1, \ldots, W_k$ and $Z_1, \ldots, Z_{\ell}$ on some complete probability space $(\Omega^*, \mathcal{F}^*, \mathbbm{Q})$ such that
\begin{eqnarray}\label{double sequence}
\mathbbm{Q}\left[\prod_{i=1}^k W_i^{n_{1i}}\prod_{j=1}^{\ell} Z_j^{n_{2j}}\right]=F(\mathbf{n_1},\,
\mathbf{n_2}).
\end{eqnarray}
Define $V_{\mu,\,t}$ as a $M_1\(E^2\)$-valued random variable satisfying
{\rm
\begin{equation}\label{eq:V}
\begin{split}
&V_{\mu,\,t}\(\otimes_{i=1}^k\ind_{C_i}^{\otimes n_{1i}}\otimes_{i=1}^{\ell}\ind_{D_i}^{\otimes n_{2i}}\)
=\prod_{i=1}^k W_i^{n_{1i}}\prod_{j=1}^{\ell} Z_j^{n_{2j}}.
\end{split}
\end{equation}
}
Combining with (\ref{F_{n,m}}), (\ref{double sequence}) and (\ref{eq:V}), we obtain the result.

\end{proof}

\begin{proposition}\label{cor0}
For each $\mu\in M_1\(E^2\)$ and $t\geq0$, the $M_1\(E^2\)$-valued random variable $V_{\mu,\,t}$
has the following properties:
$$V_{\mu,\,t}\(\(C_1\cup C_2\)\times D\)=V_{\mu,\,t}\({C_1\times D}\)+V_{\mu,\,t}\({C_2\times D}\);$$
$$V_{\mu,\,t}\(C\times \(D_1\cup D_2\)\)=V_{\mu,\,t}\(C\times D_1\)+V_{\mu,\,t}\(C\times D_2\),$$
where $C_1,\,C_2,\,D_1,\,D_2,\,C,\,D\in\mathcal{D}$ with $C_1\cap C_2=D_1\cap D_2=\emptyset$. Moreover, we have $V_{\mu,\,t}\(\emptyset\times\emptyset\)=0$ and $V_{\mu,\,t}\(E\times E\)=1.$
\end{proposition}

\begin{proof}
Note that
\begin{equation*}
\begin{split}
&\left[V_{\mu,\,t}\(\(C_1\cup C_2\)\times D\)-V_{\mu,\,t}\({C_1\times D}\)-V_{\mu,\,t}\({C_2\times D}\)\right]^2\\
&\,=V_{\mu,\,t}^2\(\(C_1\cup C_2\)\times D\)+V_{\mu,\,t}^2\({C_1\times D}\)+V_{\mu,\,t}^2\({C_2\times D}\)\\
&\qquad-2V_{\mu,\,t}\(\(C_1\cup C_2\)\times D\)\times V_{\mu,\,t}\({C_1\times D}\)\\
&\qquad-2V_{\mu,\,t}\(\(C_1\cup C_2\)\times D\)\times V_{\mu,\,t}\({C_2\times D}\)\\
&\qquad+V_{\mu,\,t}\({C_1\times D}\)\times V_{\mu,\,t}\({C_2\times D}\)\\
&\qquad+V_{\mu,\,t}\({C_2\times D}\)\times V_{\mu,\,t}\({C_1\times D}\).
\end{split}
\end{equation*}
By the coupling property, we can associate these terms above with seven $\(\Xi,\text{A},u_{\mathbbm{1}},u_{\mathbbm{2}}\)$-CPGMs
$\(Y_i\(t\),\eta\(t\)\)_{t\geq 0}$, $i=1,\ldots,7$ driven by a $\(\Xi,u_{\mathbbm{1}},u_{\mathbbm{2}}\)$-CPGM.
Put $\eta\(0\)=\{\mathbbm{1,1,2,2}\}$ and
\begin{equation*}
\begin{split}
&Y_1\(0\)=\ind_{C_1\cup C_2}\(x_1\)\ind_{C_1\cup C_2}\(x_2\)\ind_{D}\(y_1\)\ind_{D}\(y_2\);\\
&Y_2\(0\)=\ind_{C_1}\(x_1\)\ind_{C_1}\(x_2\)\ind_{D}\(y_1\)\ind_{D}\(y_2\);\\
&Y_3\(0\)=\ind_{C_2}\(x_1\)\ind_{C_2}\(x_2\)\ind_{D}\(y_1\)\ind_{D}\(y_2\);\\
&Y_4\(0\)=-2\ind_{C_1\cup C_2}\(x_1\)\ind_{C_1}\(x_2\)\ind_{D}\(y_1\)\ind_{D}\(y_2\);\\
&Y_5\(0\)=-2\ind_{C_1\cup C_2}\(x_1\)\ind_{C_2}\(x_2\)\ind_{D}\(y_1\)\ind_{D}\(y_2\);\\
&Y_6\(0\)=\ind_{C_1}\(x_1\)\ind_{C_2}\(x_2\)\ind_{D}\(y_1\)\ind_{D}\(y_2\);\\
&Y_7\(0\)=\ind_{C_2}\(x_1\)\ind_{C_1}\(x_2\)\ind_{D}\(y_1\)\ind_{D}\(y_2\).
\end{split}
\end{equation*}
It follows from 
Proposition \ref{cor1} that
\begin{equation*}
\begin{split}
&\mathbbm{Q}\left[V_{\mu,\,t}\(\(C_1\cup C_2\)\times D\)-V_{\mu,\,t}\({C_1\times D}\)-V_{\mu,\,t}\({C_2\times D}\)\right]^2\\
&\,=\sum_{i=1}^7\mathbbm{P}\left[\mathcal{G}_{\mu}\(Y_i\(t\),\eta\(t\)\)\right]\\
&\,=\mathbbm{P}\left[\mathcal{G}_{\mu}\({Y}\(t\),\eta\(t\)\)\right]
\end{split}
\end{equation*}
where the last equality follows from  Proposition \ref{coupling property}, and $\({Y}\(t\),\eta\(t\)\)_{t\geq 0}$ is  also a $\(\Xi,\text{A},u_{\mathbbm{1}},u_{\mathbbm{2}}\)$-CPGMs driven by the same $(\Xi,u_{\mathbbm{1}},u_{\mathbbm{2}})$-CPGM  with
$$Y\(0\)=\sum_{i=1}^7Y_i\(0\)=0.$$ As a consequence, $\mathbbm{P}\left[\mathcal{G}_{\mu}\({Y}\(t\),\eta\(t\)\)\right]=0$ and
 we have
\begin{equation*}\label{1}
V_{\mu,\,t}\(\(C_1\cup C_2\)\times D\)=V_{\mu,\,t}\({C_1\times D}\)+V_{\mu,\,t}\({C_2\times D}\).
\end{equation*}
Similarly, one can show that
\begin{equation*}\label{2}
V_{\mu,\,t}\(C\times \(D_1\cup D_2\)\)=V_{\mu,\,t}\(C\times D_1\)+V_{\mu,\,t}\(C\times D_2\).
\end{equation*}
Choosing $Y_{\(1\),\,\(1\)}\(0\)=\ind_{\emptyset}\otimes\ind_{\emptyset}$ and $Y_{\(1\),\,\(1\)}\(0\)=\ind_{E}\otimes\ind_{E}$ in  (\ref{F_{n,m}}),
we obtain
\begin{equation}\label{eq:normal}
\mathbbm{Q}\left[V_{\mu,\,t}\(\emptyset\times\emptyset\)\right]=0\,\,\text{and}\,\,\mathbbm{Q}\left[V_{\mu,\,t}\(E\times E\)\right]=1.
\end{equation}
It follows from (\ref{eq:V}) that $0\leq V_{\mu, t}(C\times D)\leq1$  holds for any $C\times D\in\mathcal{D}\times\mathcal{D}$. Therefore,
\begin{equation*}
V_{\mu,\,t}\(\emptyset\times\emptyset\)=0\,\text{and}\,V_{\mu,\,t}\(E\times E\)=1.
\end{equation*}

\end{proof}

\begin{theorem}
Given any $\mu\in M_1\(E^2\)$, there exists a $M_1\(E^2\)$-valued Markov process $\(\mu\(t\)\)_{t\geq0}$ on some complete probability space $\(\Omega^*,\mathcal{F}^{*},\mathbbm{Q}\)$ with transition semigroup $\left(Q_t\right)_{t\geq0}$ on $\bar{B}_{b,\,p}\(M_1\(E^2\)\)\subseteq B\(M_1\(E^2\)\)$.
In particular, for any $\mathbb{F}_{f,\,\eta\,,n}\in B_{b,\,p}\(M_1\(E^2\)\)\subseteq B\(M_1\(E^2\)\)$ with $f\in B\(E^n\)$
and $\eta\in\mathbbm{S}^n$,
\begin{equation}\label{eq-dual-2}
\begin{split}
Q_t\mathbb{F}_{f,\,\eta,\,n}\(\mu\)&=\int_{M_1\(E^2\)}\mathbb{F}_{f,\,\eta,\,n}\(v\)Q_t\(\mu,dv\)\\
&=\mathbbm{P}_{(f,\,\eta)}\left[\<{\mu}_{\eta(t)}, Y(t)\>\right]\\
&=\mathbbm{P}_{(f,\,\eta)}\mathbb{F}_{Y(t),\,\eta(t),\,|\eta(t)|}(\mu),
\end{split}
\end{equation}
where $\(Y\(t\),\eta\(t\)\)_{t\geq 0}$ is a $\(\Xi,\text{A},u_{\mathbbm{1}},u_{\mathbbm{2}}\)$-CPGM with initial value $\(Y\(0\),\eta\(0\)\)=\(f,\eta\)$.
The $M_1\(E^2\)$-valued Markov process $\(\mu\(t\)\)_{t\geq 0}$ is  {\it the generalized stepping stone model with $\Xi$-resampling mechanism.}

\end{theorem}

\begin{proof}
Initial value and transition semigroup are sufficient to determine the finite dimensional distributions of a $M_1\(E^2\)$-valued Markov process. We carry out the proof in two steps.

{\bf (i) Existence of the measure.} We show that for each $\mu\in M_1\(E^2\)$ and $t\geq0$, there
exists a random measure $\bar{V}_{\mu,\,t}\in M_1\(E^2\)$ with distribution function  denoted by $p_t\(\mu,\cdot\)$.

For any $f\in C\(E^2\)$, let $\(f_n\)_{n\geq 1}$ be a sequence of step functions with
\begin{equation*}
\begin{split}
f_n\(x,y\)=&\sum_{i=0}^{2^n-2}\sum_{j=0}^{2^n-2}f\(\frac{i}{2^n},\frac{j}{2^n}\)
\ind_{\left[\frac{i}{2^n},\,\frac{i+1}{2^n}\right)\times \left[\frac{j}{2^n},\,\frac{j+1}{2^n}\right)}\(x,y\)\\
&\quad+\sum_{j=0}^{2^n-2}f\(\frac{2^n-1}{2^n},\frac{j}{2^n}\)
\ind_{\left[\frac{2^n-1}{2^n},\,1\right]\times \left[\frac{j}{2^n},\,\frac{j+1}{2^n}\right)}\(x,y\)\\
&\quad+\sum_{i=0}^{2^n-2}f\(\frac{i}{2^n},\frac{2^n-1}{2^n}\)
\ind_{\left[\frac{i}{2^n},\,\frac{i+1}{2^n}\right)\times \left[\frac{2^n-1}{2^n},\,1\right]}\(x,y\)\\
&\quad+f\(\frac{2^n-1}{2^n},\frac{2^n-1}{2^n}\)\ind_{\left[\frac{2^n-1}{2^n},\,1\right]\times \left[\frac{2^n-1}{2^n},\,1\right]}\(x,y\).
\end{split}
\end{equation*}
Clearly, $\(f_n\)_{n\geq 1}$ uniformly converges to $f$. Define $\Lambda\,:\,C\(E^2\)\rightarrow\mathbbm{R}$
such that
\begin{equation*}
\begin{split}
\Lambda\(f\):=\lim_{n\rightarrow\infty}\Lambda\(f_n\):=\lim_{n\rightarrow\infty}\left[I_1\(n\)+I_2\(n\)+I_3\(n\)+I_4\(n\)\right]
\end{split}
\end{equation*}
with
\begin{equation*}
\begin{split}
I_1\(n\)=&\sum_{i=0}^{2^n-2}\sum_{j=0}^{2^n-2}f\(\frac{i}{2^n},\frac{j}{2^n}\)
V_{\mu,\,t}\(\left[\frac{i}{2^n},\,\frac{i+1}{2^n}\right)\times \left[\frac{j}{2^n},\,\frac{j+1}{2^n}\right)\);\\
I_2\(n\)=&\sum_{j=0}^{2^n-2}f\(\frac{2^n-1}{2^n},\frac{j}{2^n}\)
V_{\mu,\,t}\(\left[\frac{2^n-1}{2^n},\,1\right]\times \left[\frac{j}{2^n},\,\frac{j+1}{2^n}\right)\);\\
I_3\(n\)=&\sum_{i=0}^{2^n-2}f\(\frac{i}{2^n},\frac{2^n-1}{2^n}\)
V_{\mu,\,t}\(\left[\frac{i}{2^n},\,\frac{i+1}{2^n}\right)\times \left[\frac{2^n-1}{2^n},\,1\right]\);\\
I_4\(n\)=&f\(\frac{2^n-1}{2^n},\frac{2^n-1}{2^n}\)V_{\mu,\,t}\(\left[\frac{2^n-1}{2^n},\,1\right]\times \left[\frac{2^n-1}{2^n},\,1\right]\).
\end{split}
\end{equation*}

Combining the uniform continuity of function $f$ with Proposition \ref{cor0}, one can show that for any positive integer $m$,
$$|\Lambda\(f_{n+m}\)-\Lambda\(f_{n}\)|\leq \epsilon$$
holds for all $n$ large enough, namely, $\(\Lambda\(f_{n}\)\)_{n\geq 1}$ is a Cauchy sequence. The limit of the sequence exists and  $\Lambda$ is well-defined on $C\(E^2\)$.
The linear property of $\Lambda$ can be easily obtained from its definition. For any $f,\,g\in C\(E^2\)$ with
$$\| f-g\|=\sup_{\(x,\,y\)\in E^2}\mid f\(x,\,y\)-g\(x,\,y\)\mid\leq \epsilon,$$
we can choose the value of $n$ large enough and obtain
\begin{equation*}
\begin{split}
&\mid\Lambda\(f\)-\Lambda\(g\)\mid\\
&\leq\mid\Lambda\(f\)-\Lambda\(f_n\)\mid+\mid\Lambda\(f_n\)-\Lambda\(g_n\)\mid+\mid\Lambda\(g_n\)-\Lambda\(g\)\mid\\
&\leq3\epsilon.
\end{split}
\end{equation*}
In conclusion, $\Lambda:\,C\(E^2\)\rightarrow\mathbbm{R}$ is a continuous linear functional satisfying $\Lambda\(f\)\geq 0$ whenever $f\geq 0$.
By Riesz representation theorem, there exits a unique random probability measure $\bar{V}_{\mu,\,t}$
on $\sigma\(\mathcal{E}\times\mathcal{E}\)$
such that
$$\Lambda\(f\)=\int_{E^2}f d\bar{V}_{\mu,\,t}.$$
Define $p_t\(\mu,\cdot\)$ as the distribution of $\bar{V}_{\mu,\,t}$
for any $t\geq 0$.

{\bf (ii) Existence of semigroup $\(Q_t\)_{t\geq 0}$ on $\bar{B}_{b,\,p}\(M_1\(E^2\)\)$.} Denote by
\begin{equation}\label{eq:Q_t}
Q_t\mathbbm{F}\(\mu\)=\int_{M_1\(E^2\)}\mathbbm{F}\(\nu\)p_t\(\mu,d\,\nu\)
\end{equation}
for any $\mathbbm{F}\in B\(M_1\(E^2\)\)$. We first verify that there exists enough $\mathbbm{F}$ such that the value of (\ref{eq:Q_t}) can be
represented from the moments of the dual  $\(\Xi,\text{A},u_{\mathbbm{1}},u_{\mathbbm{2}}\)$-CPGMs. Then we discuss the semigroup property.

(a) {\it Dual representation.}
For any $n\in[\infty]$,
$\mathcal{D}^n$ is a $\pi$-system on $E^{n}$.
Let $\mathcal{H}$ be the collection of functions  $f\in {B}\(E^{n}\)$ with $n\in[\infty]$ satisfying
\begin{equation*}\label{eq-dual-4}
\int_{M_1\(E^2\)}\mathbb{F}_{f,\,\eta,\,n}\(v\)p_t\(\mu,dv\)=\mathbbm{P}_{(f,\,\eta)}\left[\<{\mu}_{\eta(t)}, Y(t)\>\right],
\end{equation*}
where $({Y}(t), \eta(t))_{t\geq0}$ is the dual
$\(\Xi,\text{A},u_{\mathbbm{1}},u_{\mathbbm{2}}\)$-CPGM
with initial value $\(f,\eta\)$ satisfying $\eta\in{\mathbbm{S}}^n$.
It follows from (\ref{double sequence}) that
$$\otimes_{i=1}^k\ind_{C_i}^{\otimes n_{1i}}\otimes_{j=1}^{\ell}\ind_{D_j}^{\otimes n_{2j}}\in\mathcal{H}.$$
By the linearity of expectation and integration we also have
$ah+bg\in\mathcal{H}$ for any $h,g\in\mathcal{H}$ and $a,b\in\mathbbm{R}$.
It is obvious that
$\ind_{E^n}\in\mathcal{H}$. If $h_m\in\mathcal{H},m\geq1,$ $0\leq h_m\uparrow h$ and $h$ is bounded,
let $(Y_m(s), \eta(s))_{s\leq t}$ and $(Y(s), \eta(s))_{s\leq t}$ be the associated dual processes driven by a $(\Xi,u_{\mathbbm{1}},u_{\mathbbm{2}})$-CPGM with initial values $(h_m, \eta)$ and $(h, \eta)$, respectively.  Observing that $Y_m(t)\uparrow Y(t)$ as $m\rightarrow\infty$, then by the dominated convergence theorem, we have
\begin{equation*}
\begin{split}
\int_{M_1\(E^2\)}\mathbb{F}_{h,\,\eta,\,n}\(v\)p_t\(\mu,dv\)
&=\lim_{m\rightarrow\infty} \int_{M_1\(E^2\)}\mathbb{F}_{h_m,\,\eta,\,n}\(v\)p_t\(\mu,dv\)  \\
&=\lim_{m\rightarrow\infty} \mathbbm{P}_{(h_m,\,\eta)}\left[\langle \mu_{\eta(t)}, Y_m(t)\rangle \right] \\
&=\mathbbm{P}_{(h,\,\eta)}\left[\<{\mu}_{\eta(t)}, Y(t)\>\right].
\end{split}
\end{equation*}
Subsequently,  $h\in\mathcal{H}$.
By (\ref{double sequence}) we know that
{\rm $\ind_{C}\in\mathcal{H}$} for any $ C\in\mathcal{D}^n$. Applying 
the  monotone class theorem, $\mathcal{H}$ contains all of the $\sigma\(\mathcal{D}^n\)$-measurable real-valued functions and 
therefore, (\ref{eq-dual-2})
holds for any $\mathbb{F}_{f,\,\eta,\,n}\in {B_{b,\,p}}\(M_1\(E^2\)\)$.

(b) {\it Semigroup property.}
 By the Markov property of the dual $\(\Xi,\text{A},u_{\mathbbm{1}},u_{\mathbbm{2}}\)$-CPGM $(Y(s),\eta(s))_{s\geq 0}$, one can see that given any
$\mathbbm{F}_{f,\,\eta,\,n}\in B_{b,\,p}\(M_1\(E^2\)\)$ and $\mu\in M_1\(E^2\)$, $\forall\,s,\,t\geq 0$,
\begin{equation*}
\begin{split}
Q_sQ_t\mathbbm{F}_{f,\,\eta,\,n}(\mu)
&=\mathbbm{P}_{(f,\,\eta)}\left[Q_s\mathbbm{F}_{Y(t),\,\eta\(t\),\,|\eta\(t\)|}(\mu)\right]\\
&=\mathbbm{P}_{(f,\,\eta)}\left[\mathbbm{F}_{Y(s+t),\,\eta\(s+t\),\,|\eta\(s+t\)|}(\mu)\right]\\
&=Q_{s+t} \mathbbm{F}_{f,\,\eta,\,n}(\mu).
\end{split}
\end{equation*}
For general $\mathbbm{F}\in \bar{B}_{b,\,p}\(M_1\(E^2\)\)$, there exists an approximating sequence
$\{\mathbbm{F}_{k}, k\geq1\}\subseteq B_{b,\,p}\(M_1\(E^2\)\)$ such that
$$\lim_{k\rightarrow\infty}\|\mathbbm{F}_{k}-\mathbbm{F}\|=\lim_{k\rightarrow\infty}\sup_{\mu\in M_1\(E^2\)}\|\mathbbm{F}_{k}\(\mu\)-\mathbbm{F}\(\mu\)\|=0.$$
For any $t\geq 0$, 
we have
\begin{equation*}
\begin{split}
|Q_t\mathbbm{F}\(\mu\)-Q_t\mathbbm{F}_{k}\(\mu\)|
&\,=\int_{M_1\(E^2\)}\(\mathbbm{F}\(\nu\)-\mathbbm{F}_{k}\(\nu\)\)p_t\(\mu,d\,\nu\)\\
&\,\leq\|\mathbbm{F}-\mathbbm{F}_{k}\|\rightarrow 0\text{~~as~~}k\rightarrow\infty.
\end{split}
\end{equation*}
That is to say
$Q_t\mathbbm{F}\(\mu\)=\lim_{k\rightarrow\infty}Q_t\mathbbm{F}_{k}\(\mu\).$
Therefore, for any $s,\,t\geq 0$ and $\mathbbm{F}\in \bar{B}_{b,\,p}\(M_1\(E^2\)\)$, we have
\begin{equation*}
\begin{split}
Q_sQ_t\mathbbm{F}(\mu)
&\,=\lim_{k\rightarrow\infty}Q_sQ_t\mathbbm{F}_{k}(\mu)=\lim_{k\rightarrow\infty}Q_{s+t}\mathbbm{F}_{k}(\mu)
=Q_{s+t} \mathbbm{F}(\mu).
\end{split}
\end{equation*}
The semigroup property then follows.


\end{proof}

\begin{proposition}
 $(Q_t)_{t\geq 0}$ is a Feller semigroup on $C(M_1(E^2))$.
\end{proposition}

\begin{proof}
We need to verify the required properties point by point.

(a) {\it Semigroup property.} Since $C\(M_1\(E^2\)\)\subseteq \bar{B}_{b,\,p}\(M_1\(E^2\)\)$, the restriction of $\(Q_t\)_{t\geq 0}$
on $C\(M_1\(E^2\)\)$ satisfies the semigroup property.

(b) {\it Strongly continuous property.}
Given $\mathbbm{F}_{f,\,\eta,\,n}\in C_{p}\(M_1\(E^2\)\)$, let's recall that $\tau_1$ is the first jumping time of the dual $\(\Xi,\text{A},u_{\mathbbm{1}},u_{\mathbbm{2}}\)$-CPGM with initial value $\(f,\eta\)$. Thus $\tau_1\sim \exp\(\lambda\)$ with
$\lambda=|\eta|_{\mathbbm{2}}u_{\mathbbm{1}}+|\eta|_{\mathbbm{1}}u_{\mathbbm{2}}+\lambda_{|\eta|_{\mathbbm{2}}}+\lambda_{|\eta|_{\mathbbm{1}}}$.
The probability  of $\{\tau_1<t\}$
(that is either coalescence or migration has occurred by time $t$) converges to $0$ as $t\rightarrow 0$. Then
\begin{equation*}\label{eq-expand}
\begin{split}
&\lim_{t\rightarrow 0}| Q_t\mathbbm{F}_{f,\,\eta,\,n}(\mu)-\mathbbm{F}_{f,\,\eta,\,n}(\mu)|=\lim_{t\rightarrow 0}|e^{-\lambda t}\mathcal{G}_{\mu}\(T_t^{\(n\)}f,\eta\)-\mathcal{G}_{\mu}\(f,\eta\)|=0.
\end{split}
\end{equation*}
For general $\mathbbm{F}\in C\(M_1\(E^2\)\)$, we reapply the approximating sequence
$\{\mathbbm{F}_{k}, k\geq1\}\subseteq C_{p}\(M_1\(E^2\)\)$ to get
\begin{equation*}
\begin{split}
&\lim_{t\rightarrow 0}| Q_t\mathbbm{F}(\mu)-\mathbbm{F}(\mu)|\\
&\leq\lim_{k\rightarrow\infty}\lim_{t\rightarrow 0}\(|Q_t\mathbbm{F}(\mu)-Q_t\mathbbm{F}_{k}(\mu)|+|Q_t\mathbbm{F}_{k}(\mu)-\mathbbm{F}_{k}(\mu)|
+|\mathbbm{F}_{k}(\mu)-\mathbbm{F}(\mu)|\)\\
&=0\end{split}
\end{equation*}
Therefore $Q_t$ is strongly continuous.

(c) {\it Contraction operator.}
\begin{equation*}
\begin{split}
\| Q_t\mathbbm{F}\|=&\sup_{\mu}|Q_t\mathbbm{F}\(\mu\)|
=\sup_{\mu}\left|\int_{M_1\(E^2\)}\mathbbm{F}\(\nu\)p_t\(\mu,d\nu\)\right|
\leq \| \mathbbm{F}\|.
\end{split}
\end{equation*}
Consequently, $Q_t$ is a contraction operator.

(d) {\it $Q_t\,:\,C\(M_1\(E^2\)\)\rightarrow C\(M_1\(E^2\)\)$.} For any  $\mathbbm{F}\in C\(M_1\(E^2\)\)$,
$$Q_t\mathbbm{F}\(\mu\)=\lim_{k\rightarrow\infty}Q_t\mathbbm{F}_{k}\(\mu\),$$
we know that $Q_t\mathbbm{F}\(\mu\)$ is a limit point of a sequence belonging to $C_p\(M_1\(E^2\)\)$.
Thus, $Q_t\mathbbm{F}\(\mu\)\in C\(M_1\(E^2\)\)=\overline{C_p\(M_1\(E^2\)\)}$.

Combining with (a)-(d), one can see that $\(Q_t\)_{t\geq 0}$  is a Feller semigroup.
\end{proof}

Let $\mathcal{L}^{*}$ be the generator of the generalized stepping stone model with $\Xi$-resampling mechanism $(\mu(t))_{t\geq 0}$. For any $\mathbb{F}_{f,\,\eta\,,n}\in B_{b,\,p}\(M_1\(E^2\)\)$ with $f\in B\(E^n\)$
and $\eta\in\mathbbm{S}^n$,  one can see that
$$\mathcal{L}^{*}\mathbb{F}_{f,\,\eta,\,n}\(\mu\)=\lim_{t\rightarrow 0}\frac{Q_t\mathbb{F}_{f,\,\eta,\,n}\(\mu\)-\mathbb{F}_{f,\,\eta,\,n}\(\mu\)}{t}, \quad \mu\in M_1\(E^2\)$$
whenever the limit exists. What follows is a proposition to represent  $\mathcal{L}^{*}$ from the generator of the dual $\(\Xi,\text{A},u_{\mathbbm{1}},u_{\mathbbm{2}}\)$-CPGM. Recall $\mathcal{G}_{\mu} $ and $\mathcal{L}$ defined in (\ref{G_{f,u}}) and (\ref{generator-dual}), respectively.

\begin{proposition}
The generator $\mathcal{L}^{*}$ of the generalized stepping stone model with $\Xi$-resampling mechanism $(\mu(t))_{t\geq 0}$ with initial value $\mu\(0\)=\mu$ is of the form below
\begin{equation}\label{generator}
\begin{split}
&\mathcal{L}^{*}\mathbb{F}_{f,\,\eta,\,n}\(\mu\)
=\mathcal{L}\mathcal{G}_{\mu}\(f,\eta\)
\end{split}
\end{equation}
where $\mathbb{F}_{f,\,\eta\,,n}\in B_{b,\,p}\(M_1\(E^2\)\)$ with $\(f,\,\eta\)\in B\(E^n\)\times \mathbbm{S}^n$.
\end{proposition}
\begin{proof}
Applying (\ref{eq-dual-2}), one can see that
\begin{equation*}
\begin{split}
\mathcal{L}^{*}\mathbb{F}_{f,\,\eta,\,n}\(\mu\)=&\lim_{t\rightarrow0}\frac{Q_t\mathbb{F}_{f,\,\eta,\,n}\(\mu\)-\mathbb{F}_{f,\,\eta,\,n}\(\mu\)}{t}\\
=&\lim_{t\rightarrow0}\frac{\mathbbm{P}\left[\mathcal{G}_{\mu}\(Y\(t\),\eta\(t\)\)\right]-\mathcal{G}_{\mu}\(f,\eta\)}{t}\\
=&\mathcal{L}\mathcal{G}_{\mu}\(f,\eta\).
\end{split}
\end{equation*}

\end{proof}

\section{Stationary distribution of the generalized stepping stone model}\label{sec5}
In this section we discuss the limit stationary distribution for the generalized stepping stone model with $\Xi$-resampling mechanism under the condition below.

\noindent{\bf Condition A:} The mutation operator
$A$ generates an irreducible semigroup $\(T_t\)_{t\geq 0}$ and $\tilde{\pi}$ is the unique invariant measure in ${M}_1(E)$ such that $T^{*}_t\nu\rightarrow\tilde{\pi}$ weakly for any $\nu\in{M}_1\(E\)$
as $t\rightarrow\infty$, where $T_t^{*}$ is the adjoint for $T_t$.

\begin{theorem}
Under Condition A, the generalized stepping stone model with $\Xi$-resampling mechanism $\({\mu}(t)\)_{t\geq 0}$ has a unique
stationary distribution $\Pi\in {M}_1({M}_1(E^2))$ such that
\begin{equation}\label{stationary distribution}
\int_{{M}_1(E^2)}\<\mu_{\xi}, f\>\Pi(d{\mu})=
\<\tilde{\pi}, f\>,
\end{equation} for any $\xi\in \mathbbm{S}$ and $f\in B\(E\)$.
\end{theorem}

\begin{proof}
The result of the current theorem for the $\Xi$-Fleming-Viot process and the classical stepping stone model was  proved
in  \cite{LLXZ} and \cite{Handa_2}, respectively.
Our proof is an adaption  of their work.

We first show the existence of a stationary distribution for the generalized stepping stone model.
As $\(T_t^*, T_t^*\)\mu\rightarrow\(\tilde{\pi},\tilde{\pi}\)$, the family $\{\(T_t^*,T_t^*\)\mu:\ t\ge
0\}$ is pre-compact, and hence, tight in ${M}_1(E^2)$.
Thus, for any
$\ep>0$, there exists a compact subset $K_\ep$ of $E^2$ such that
$\(T_t^*,T_t^*\)\mu\(K^c_\ep\)<\ep$ for all $t\ge 0$. Let
\[\KK_\ep=\left\{\mu\in {{M}}_1(E^2):\;\mu(K^c_{\ep k^{-1}2^{-k}})\le
k^{-1},\;\;\forall\ k\ge 1\right\}.\] For any $\de>0$, choose  $k\ge 1$
be such that $k^{-1}<\de$. Then for all $\mu\in\KK_\ep$, we have
\[\mu(K^c_{\ep k^{-1}2^{-k}})\le
k^{-1}<\de,\] and hence, $\KK_\ep$ is tight in ${M}_1(E^2)$. Then
$\KK_\ep$ is a pre-compact subset of ${M}_1(E^2)$.
Note that
\begin{eqnarray*}
t^{-1}\int^t_0\mathbbm{Q}\left[\mu^{-1}\(s\)\right]ds(\KK^c_\ep):&=&t^{-1}\int^t_0\mathbbm{Q}\left[\mu\(s\)\in\KK^c_\ep\right]ds
\\&=&t^{-1}\int^t_0\mathbbm{Q}\left[\exists\
k\ge
1,\;\;\mu\(s\)(K^c_{\ep k^{-1}2^{-k}})>k^{-1}\right]ds\\
&\le&t^{-1}\int^t_0\sum^\infty_{k=1}k\mathbbm{Q}\left[ \mu\(s\)(K^c_{\ep k^{-1}2^{-k}})\right]ds\\
&=&t^{-1}\int^t_0\sum^\infty_{k=1}k \(T_s^*, T_s^*\)\mu(K^c_{\ep k^{-1}2^{-k}})ds\\
&<&t^{-1}\int^t_0\sum^\infty_{k=1}k \ep k^{-1}2^{-k}ds=\ep.
\end{eqnarray*}
Thus, the family  $\left\{t^{-1}\int^t_0\mathbbm{Q}[\mu^{-1}\(s\)]ds:\;t\ge
0\right\}$ is tight, and hence, pre-compact in ${M}_1({M}_1(E^2))$. Let
$\Pi$ be a limit point. Then there exists a sequence $(t_n)$ such
that $t_n\uparrow \infty$ and
\[\lim_{n\goto\infty }t_n^{-1}\int^{t_n}_0\mathbbm{Q}[\mu^{-1}\(s\)]ds=\Pi.\]
For any $r\ge 0$,
\begin{eqnarray*}
\Pi [\mu^{-1}\(r\)]&=&\lim_{n\to\infty}t_n^{-1}\int^{t_n}_0\mathbbm{Q}[\mu^{-1}\(s\)]\circ[\mu^{-1}\(r\)]
ds\\
&=&\lim_{n\to\infty}t_n^{-1}\int^{t_n+r}_r\mathbbm{Q}[\mu^{-1}\(s\)]ds\\
&=&\Pi.
\end{eqnarray*}
Namely, $\Pi$ is an invariant measure of the stochastic process
$\(\mu\(t\)\)_{t\geq 0}$.

It follows from the construction of the dual $\(\Xi,\text{A},u_{\mathbbm{1}},u_{\mathbbm{2}}\)$-CPGM that
$|\eta(t)|$ is non-increasing with respect to $t$ and
$\lim_{t\rightarrow\infty}|\eta(t)|=1$. Denote by $\tau=\inf\{t\geq0: |\eta\(t\)|=1\}$.
Note that $\tau<\infty$ a.s. and $\(Y\(t\), \eta\(t\)\)\in {B}(E)\times {\mathbbm{S}}$ for any $t\geq\tau$. One can see that
\begin{equation}\label{limit state distribution}
\begin{split}
\int_{M_1\(E^2\)}\left<\mu_{\eta},f\right>\Pi\(d\mu\)
=&\lim_{t\rightarrow\infty}\mathbbm{Q}_{\Pi}\left<\mu\(t\)_{\eta},f\right>\\
=&\lim_{t\rightarrow\infty}\int_{M_1\(E^2\)}\mathbbm{P}_{\(f,\,\eta\)}\left<\mu_{\eta\(t\)},Y\(t\)\right>\Pi\(d\mu\)\\
=&\lim_{t\rightarrow\infty}\int_{M_1\(E^2\)}\mathbbm{P}_{\(f,\,\eta\)}\left[\left<\mu_{\eta\(t\)},Y\(t\)\right>{\ind_{\{\tau\leq t\}}}\right]\Pi\(d\mu\)\\
=&\lim_{t\rightarrow\infty}\int_{M_1\(E^2\)}\mathbbm{P}_{\(f,\,\eta\)}\mathbbm{P}_{\(Y\(\tau\),\,\eta\(\tau\)\)}\left[\left<\mu_{\eta\(t-\tau\)},Y\(t-\tau\)\right>{\ind_{\{\tau\leq t\}}}\right]\Pi\(d\mu\)\\
=&\lim_{t\rightarrow\infty}\int_{M_1\(E^2\)}\mathbbm{P}_{\(f,\,\eta\)}\mathbbm{P}_{\(Y\(\tau\),\,\eta\(\tau\)\)}\left<\mu_{\eta\(t\)},Y\(t\)\right>\Pi\(d\mu\)\\
=&\lim_{t\rightarrow\infty}\int_{M_1\(E^2\)}\mathbbm{P}_{\(f,\,\eta\)}\left<\mu_{\eta\(t+\tau\)},T_tY\(\tau\)\right>\Pi\(d\mu\)\\
=&\lim_{t\rightarrow\infty}\int_{M_1\(E^2\)}\mathbbm{P}_{\(f,\,\eta\)}\left<T_t^{*}\mu_{\eta\(t+\tau\)},Y\(\tau\)\right>\Pi\(d\mu\)\\
=&\mathbbm{P}_{\(f,\,\eta\)}\left[\left<\tilde{\pi},Y\(\tau\)\right>\right].
\end{split}
\end{equation}
Thus, (\ref{stationary distribution}) follows from (\ref{limit state distribution}) with
$n=1$.
\end{proof}

\section{Reversibility of the generalized stepping stone model}\label{sec6}
In this section we discuss the reversibility of the generalized stepping stone model with $\Xi$-resampling mechanism. 
In view of the well-known results in the literature, we make the following
assumption on the mutation operator.

\noindent{\bf Condition B:} The mutation operator
$A$ is of the uniform type, i.e.,
$$Af(x)=\frac{\theta}{2}\int_E\(f(y)-f(x)\)\nu_0(dy)$$
for some $\theta>0$, $\nu_0\in M_1{\(E\)}$ and $f\in B\(E\)$.

\begin{definition}
The generalized stepping stone model with $\Xi$-resampling mechanism $({\mu}(t))_{t\geq0}=(\mu_{\mathbbm{1}}(t),\mu_{\mathbbm{2}}(t))_{t\geq0}$ with stationary
distribution $\Pi$ is {\it reversible} if  for any
nonnegative integers $n, m, p$ and $q$, we have
\begin{equation}\label{eq-rev}
\begin{split}
&\mathbbm{Q}_{\Pi}\left[\(\<{{\mu_{\mathbbm{1}}}(t)}^{\otimes n}, f^{{\mathbbm{1}}}\>\<{{\mu_{\mathbbm{2}}}(t)}^{\otimes
m},
f^{{\mathbbm{2}}}\>\)\mathcal{L}^{*}\(\<{{\mu_{\mathbbm{1}}}(t)}^{\otimes p}, g^{{\mathbbm{1}}}\>\<{{\mu_{\mathbbm{2}}}(t)}^{\otimes q}, g^{{\mathbbm{2}}}\>\)\right]\\
&=\mathbbm{Q}_{\Pi}\left[\(\<{\mu_{\mathbbm{1}}(t)}^{\otimes p}, g^{{\mathbbm{1}}}\>\<{{\mu_{\mathbbm{2}}}(t)}^{\otimes
q}, g^{{\mathbbm{2}}}\>\)\mathcal{L}^{*}\(\<{{\mu_{\mathbbm{1}}}(t)}^{\otimes n}, f^{{\mathbbm{1}}}\>\<{{\mu_{\mathbbm{2}}}(t)}^{\otimes
m}, f^{{\mathbbm{2}}}\>\)\right],
\end{split}
\end{equation}
where $f^{{\mathbbm{1}}}\in {B}(E^n)$, $f^{{\mathbbm{2}}}\in {B}(E^m)$, $g^{{\mathbbm{1}}}\in
{B}(E^{p})$, $g^{{\mathbbm{2}}}\in {B}(E^{q})$ and $\mathcal{L}^{*}$ is the generator given by $(\text{\ref{generator}})$.
Without loss of generality, we denote by $\<{\mu_{\mathbbm{i}}(t)}^{\otimes 0}, \cdot\>=1$ with $\mathbbm{i}\in\{\mathbbm{1,2}\}$ and $\mathcal{L}^{*}1=0$.
\end{definition}
\begin{theorem}\label{reversibility theorem}
Under Condition B, the generalized stepping stone model with $\Xi$-resampling mechanism is not reversible.
\end{theorem}
\begin{proof}
The proof is based on computations  of moments and joint moments of different orders. The software {\it Maple} is used to
carry out symbolic calculations. We sketch the main steps and refer to Appendix for more details.

Let $E^{*}$ be any subset of $E$ with $0<\nu_0(E^{*})=\alpha<1$.  We start with $(p,q)=(0,0)$. By choosing different values of $n$, $m$,
$f^{\mathbbm{1}}$ and $f^{\mathbbm{2}}$ in (\ref{eq-rev}), expressions for the moments and joint moments
\begin{equation*}
\begin{split}
&\mathbbm{Q}\left[{\mu_{\mathbbm{1}}}(E^{*})\right],\,\,\mathbbm{Q}\left[{\mu_{\mathbbm{2}}}(E^{*})\right],\,\,\mathbbm{Q}\left[{\mu_{\mathbbm{1}}}^2(E^{*})\right],\,\,\mathbbm{Q}\left[{\mu_{\mathbbm{1}}}(E^{*}){\mu_{\mathbbm{2}}}(E^{*})\right],\,\,\mathbbm{Q}\left[{\mu_{\mathbbm{2}}}^2(E^{*})\right],\\
&\mathbbm{Q}\left[{\mu_{\mathbbm{1}}}^3(E^{*})\right],\,\,\mathbbm{Q}\left[{\mu_{\mathbbm{1}}}^2(E^{*}){\mu_{\mathbbm{2}}}(E^{*})\right],\,\,\mathbbm{Q}\left[{\mu_{\mathbbm{1}}}(E^{*}){\mu_{\mathbbm{2}}}^2(E^{*})\right],\,\,\mathbbm{Q}\left[{\mu_{\mathbbm{2}}}^3(E^{*})\right],\\
&\mathbbm{Q}\left[{\mu_{\mathbbm{1}}}^4(E^{*})\right],\,\,\mathbbm{Q}\left[{\mu_{\mathbbm{1}}}^3(E^{*}){\mu_{\mathbbm{2}}}(E^{*})\right],\,\,\mathbbm{Q}\left[{\mu_{\mathbbm{1}}}^2(E^{*}){\mu_{\mathbbm{2}}}^2(E^{*})\right],
\mathbbm{Q}\left[{\mu_{\mathbbm{1}}}(E^{*}){\mu_{\mathbbm{2}}}^3(E^{*})\right],\,\,\mathbbm{Q}\left[{\mu_{\mathbbm{2}}}^4(E^{*})\right]
\end{split}
\end{equation*}
are all available.

We then consider the case that $(p,q)\neq(0,0)$.
For  $$(n, m)=(1, 0),\,(p, q)=(0,1)\,\text{and}\,f^{\mathbbm{1}}=g^{\mathbbm{2}}=\ind_{E^{*}},$$ by (\ref{eq-rev})
we have that the condition  ${u_{\mathbbm 1}}={u_{\mathbbm 2}}$ is necessary for the process being reversible.
Further substituting  $$(n, m)=(2, 0),\,(p, q)=(0,1)\,\text{and}\,f^{\mathbbm{1}}=\ind_{E^{*}\times E^{*}},\,g^{\mathbbm{2}}=\ind_{E^{*}}$$ in (\ref{eq-rev}), we have that
the condition  $\al=1/2$ is necessary for the process being reversible.
Given $\al=1/2$ and ${u_{\mathbbm 1}}={u_{\mathbbm 2}}$, we choose
$$(n, m)=(1, 1),\,(p, q)=(2,0),\,f^{\mathbbm{1}}=f^{\mathbbm{2}}=\ind_{E^{*}},\,g^{\mathbbm{1}}=\ind_{E^{*}\times E^{*}}$$ and
$$(n, m)=(2, 1),\,(p, q)=(1,0),\,f^{\mathbbm{1}}=\ind_{E^{*}\times E^{*}},\,f^{\mathbbm{2}}=g^{\mathbbm{1}}=\ind_{E^{*}}$$ in (\ref{eq-rev}), respectively. Note that the two equations can not  be true at the same time and we reach  a contradiction.
Thus, the process is not reversible.
\end{proof}

\begin{appendices}
\section{The detailed calculations in the proof of Theorem \ref{reversibility theorem}}
For simplicity of notation, we write
\begin{equation*}
\begin{split}
&a_2=\la_{2;2;0},\,\,\,\,\,\,a_{21}=\la_{3;2;1},\,\,\,\,\,\,a_3=\la_{3;3;0},\\
&a_{211}=\la_{4;2;2},\,\,a_{22}=\la_{4;2,2;0},\,\,a_{31}=\la_{4;3;1}, \,\,a_4=\la_{4;4;0}.
\end{split}
\end{equation*}
Let $E^{*}$ be  a subset of the type space $E$ with $\nu_0\(E^{*}\)=\al>0$. For any positive integers  $n$ and $m$, the joint moments are defined as $$M_{n,m}=\mathbbm{Q}\left[{\mu_{\mathbbm{1}}}^n\(E^{*}\){\mu_{\mathbbm{2}}}^m\(E^{*}\)\right].$$

We begin with the discussion for the first moments, i.e., $n+m=1$. \\
For $(n,m)=(1,0),\,(p,q)=(0,0)$ and $f^{{\mathbbm{1}}}=\ind_{E^{*}}$ in (\ref{eq-rev}), we have
\begin{equation}\label{E_{11}}
\left(\frac{\theta}{2}+u_{\mathbbm 2}\right)M_{1,0}-u_{\mathbbm 2}M_{0,1}=\frac{\theta\al}{2}.
\end{equation}
For $(n,m)=(0,1),\,(p,q)=(0,0)$ and $f^{{\mathbbm{2}}}=\ind_{E^{*}}$ in (\ref{eq-rev}), we have
\begin{equation}\label{E_{12}}
-u_{\mathbbm 1}M_{1,0}+\left(\frac{\theta}{2}+u_{\mathbbm 1}\right)M_{0,1}=\frac{\theta\al}{2}.
\end{equation}
By (\ref{E_{11}}) and (\ref{E_{12}}) we have  $M_{1,0}=M_{0,1}=\al.$

Then we consider the second moments, i.e., $n+m=2$.\\
For $(n,m)=(2,0),\,(p,q)=(0,0)$ and $f^{{\mathbbm{1}}}=\ind_{E^{*}\times E^{*}}$ in (\ref{eq-rev}), we have
\begin{equation}\label{E_{21}}
\left( \theta+2\,u_{\mathbbm 2}+a_
{{2}} \right) M_{{2,0}}-2\,u_{\mathbbm 2}M_{{1,1}}=\theta\,{\alpha}^{2}+a_{{2}}\,\alpha.
\end{equation}
For $(n,m)=(0,2),\,(p,q)=(0,0)$ and $f^{{\mathbbm{2}}}=\ind_{E^{*}\times E^{*}}$ in (\ref{eq-rev}), we have
\begin{equation}\label{E_{22}}
-2\,u_{\mathbbm 1}M_{{1,1}}+\left( \theta+2\,u_{\mathbbm 1}+a_
{{2}} \right) M_{{0,2}}=\theta\,{\alpha}^{2}+a_{{2}}\,\alpha.
\end{equation}
For $(n,m)=(1,1),\,(p,q)=(0,0)$ and $f^{{\mathbbm{1}}}=f^{{\mathbbm{2}}}=\ind_{E^{*}}$ in (\ref{eq-rev}), we have
\begin{equation}\label{E_{23}}
-u_{\mathbbm 1}M_{
{2,0}}+ \left( \theta+u_{\mathbbm 2}+u_{\mathbbm 1} \right) M_{{1,1}}-u_{\mathbbm 2}M_{{0,2}}=\theta\,{\alpha}^{2}.
\end{equation}
(\ref{E_{21}})-(\ref{E_{23}}) constitute a system of linear equations. Note that the coefficient matrix of
$M_{2,0}$, $M_{1,1}$, $M_{0,2}$ is invertible. Thus, the solution is unique.
Substituting $(n,m)=(1,0),\,(p,q)=(0,1)$ and $f^{{\mathbbm{1}}}=\ind_{E^{*}},\,g^{{\mathbbm{2}}}=\ind_{E^{*}}$ in (\ref{eq-rev}), we have the reversible equation
\begin{equation}\label{S_{1}}
\left( u_{\mathbbm 1}-u_{\mathbbm 2} \right) M_{{1,1}}-u_{\mathbbm 1}M_{{2,0}}+u_{\mathbbm 2}M_{{0,2}}=0.
\end{equation}
Replacing the second moments, the numerator of (\ref{S_{1}}) equals to $0$, i.e.,
\begin{equation*}
{{\alpha\,\theta\,a_{{2}} \left(u_{\mathbbm 1}-u_{\mathbbm 2} \right)  \left( \theta+2\,u_{\mathbbm 1}
+a_{{2}}+2\,u_{\mathbbm 2} \right)  \left( \alpha-1 \right) }}=0.
\end{equation*}
Thus, a necessary condition for this process being reversible is
that ${u_{\mathbbm 1}}={u_{\mathbbm 2}}$.

Now we continue to consider the third moments, i.e., $m+n=3$.\\
For $(n,m)=(3,0),\,(p,q)=(0,0)$ and $f^{{\mathbbm{1}}}=\ind_{E^{*}\times E^{*}\times E^{*}}$ in (\ref{eq-rev}), we have
\begin{equation}\label{E_{31}}
\left( 3\,a_{{21}}+a_{{3}}+\frac32\,\theta+3\,u_{\mathbbm 2} \right) M_{{3,0}}-3\,u_{\mathbbm 2}M_{
{2,1}}=\left( 3\,a_{{21}}+\frac32\,\alpha\,\theta \right) M_{{2,0}}+a_{{3}}\alpha.
\end{equation}
For $(n,m)=(2,1),\,(p,q)=(0,0)$ and $f^{{\mathbbm{1}}}=\ind_{E^{*}\times E^{*}},\,f^{{\mathbbm{2}}}=\ind_{E^{*}}$ in (\ref{eq-rev}), we have
\begin{equation}\label{E_{32}}
-u_{\mathbbm 1}M_{{3,0}}+\left( a_{{2}}+\frac32\,\theta+2\,u_{\mathbbm 2}+u_{\mathbbm 1} \right) M_{{2,1}}-2\,u_{\mathbbm 2}M_{{1,2
}}=\left( \alpha\,\theta+a_{{2}} \right) M_{{1,1}}+\frac{\theta\,\alpha}{2}\,M
_{{2,0}}.
\end{equation}
For $(n,m)=(1,2),\,(p,q)=(0,0)$ and $f^{{\mathbbm{1}}}=\ind_{E^{*}},\,f^{{\mathbbm{2}}}=\ind_{E^{*}\times E^{*}}$ in (\ref{eq-rev}), we have
\begin{equation}\label{E_{33}}
-2
\,u_{\mathbbm 1}M_{{2,1}}+\left( a_2+\frac32\,\theta+u_{\mathbbm 2}+2\,u_{\mathbbm 1} \right) M_{{1,2}}-u_{\mathbbm 2}M_{{0,3}}=\left( \alpha\,\theta+a_{{2}} \right) M_
{{1,1}}+\frac{\theta\,\alpha}{2}\,M_{{0,2}}.
\end{equation}
For $(n,m)=(0,3),\,(p,q)=(0,0)$ and $f^{{\mathbbm{2}}}=\ind_{E^{*}\times E^{*}\times E^{*}}$ in (\ref{eq-rev}), we have
\begin{equation}\label{E_{34}}
-3\,u_{\mathbbm 1}M_{{
1,2}}+\left( 3\,a_{{21}}+a_{{3}}+\frac32\,\theta+3\,u_{\mathbbm 1}\right) M_{{0,3}}=\left( 3\,a_{{21}}+\frac32\,\alpha\,\theta \right) M_{{0,2}}+a_{{3}}\alpha.
\end{equation}
(\ref{E_{31}})-(\ref{E_{34}}) constitute a system of linear equations and the coefficient matrix of $M_{3,0},\,M_{2,1}$, $M_{1,2},\,M_{0,3}$ is invertible. Thus, the system of equations has a unique solution.
Substituting $(n,m)=(2,0),\,(p,q)=(0,1)$ and $f^{{\mathbbm{1}}}=\ind_{E^{*}\times E^{*}},\,g^{{\mathbbm{2}}}=\ind_{E^{*}}$ in (\ref{eq-rev}), we have
\begin{equation}\label{T_{1}}
\left( \alpha\,\theta+a_{{2}} \right) M_{{1,1}}-\frac{\theta\,\alpha}{2}\,M
_{{2,0}}- \left( 2\,u_{\mathbbm 2}+a_{{2}}+\frac{\theta}{2}-u_{\mathbbm 1} \right) M_{{2,1}}-u_{\mathbbm 1}M_{{3,0}}+
2\,u_{\mathbbm 2}M_{{1,2}}=0.
\end{equation}
Note that $a_2=a_{21}+a_3$. Substituting the expressions of the moments, we have the numerator of (\ref{T_{1}}) equal to $0$, i.e.,
\begin{equation*}
4\,\alpha\,{\theta}^{2}u_{\mathbbm 1} \left( \alpha-1 \right)  \left( 3\,{a_{{21}}}
^{2}+{a_{{3}}}^{2}+8\,a_{{3}}u+2\,a_{{3}}\theta+4\,a_{{3}}a_{{21}}
 \right)  \left( -1+2\,\alpha \right)=0.
\end{equation*}
So far we conclude that the process is not reversible except the case when there exists a subset $E^{*}\subseteq E$ with $\alpha=\nu_0\(E^{*}\)=1/2$.

For the case of $\al=1/2$, we need higher moments to reach a contradiction. \\
For $(n,m)=(4,0),\,(p,q)=(0,0)$ and $f^{{\mathbbm{1}}}=\ind_{E^{*}\times E^{*}\times E^{*}\times E^{*}}$ in (\ref{eq-rev}), we have
\begin{equation}\label{E_{41}}
\begin{split}
& \left( 4\,u_{\mathbbm 2}+6\,a_{{
211}}+3\,a_{{22}}+4\,a_{{31}}+a_{{4}}+2\,\theta \right) M_{{4,0}}-4\,u_{\mathbbm 2}M
_{{31}}\\
&\,\,=\left( 2\,\theta\,\alpha+6\,a_{{211}} \right) M_{{3,0}}+ \left( 3\,a_{
{22}}+4\,a_{{31}} \right) M_{{2,0}}+a_{{4}}\,\alpha.
\end{split}
\end{equation}
For $(n,m)=(3,1),\,(p,q)=(0,0)$ and $f^{{\mathbbm{1}}}=\ind_{E^{*}\times E^{*}\times E^{*}},\,f^{{\mathbbm{2}}}=\ind_{E^{*}}$ in (\ref{eq-rev}), we have
\begin{equation}\label{E_{42}}
\begin{split}
&-u_{\mathbbm 1}M_{{4,0}}+ \left( a_{{3}
}+3\,a_{{21}}+2\,\theta+3\,u_{\mathbbm 2}+u_{\mathbbm 1} \right) M_{{3,1}}-3\,u_{\mathbbm 2}M_{{2,2}}\\
&\,\,= \left( 3\,a_{{21}}+\frac{3\,\theta\,\alpha}{2} \right) M_{{2,1}}+a_{{3}}M_{{1,1
}}+\frac{\,\theta\,\alpha}{2}\,M_{{3,0}}.
\end{split}
\end{equation}
For $(n,m)=(2,2),\,(p,q)=(0,0)$ and $f^{{\mathbbm{1}}}=\ind_{E^{*}\times E^{*}},\,f^{{\mathbbm{2}}}=\ind_{E^{*}\times E^{*}}$ in (\ref{eq-rev}), we have
\begin{equation}\label{E_{43}}
\begin{split}
&-2\,u_{\mathbbm 1}M_{{3,1}}+ \left( 2\,a_{{2}}
+2\,\theta+2\,u_{\mathbbm 2}+2\,u_{\mathbbm 1} \right) M_{{2,2}}-2\,u_{\mathbbm 2}M_{{1,3}}\\
&\,\,=\left( \theta\,\alpha+a_{{2}} \right) M_{{1,2}}+ \left( \theta\,\alpha
+a_{{2}} \right) M_{{2,1}}.
\end{split}
\end{equation}
For $(n,m)=(1,3),\,(p,q)=(0,0)$ and $f^{{\mathbbm{1}}}=\ind_{ E^{*}},\,f^{{\mathbbm{2}}}=\ind_{E^{*}\times E^{*}\times E^{*}}$ in (\ref{eq-rev}), we have
\begin{equation}\label{E_{44}}
 \begin{split}
&-3\,u_{\mathbbm 1}M_{{2,2}}+ \left( 2\,
\theta+u_{\mathbbm 2}+3\,u_{\mathbbm 1}+3\,a_{{21}}+a_{{3}} \right) M_{{1,3}}-u_{\mathbbm 2}M_{{0,4}}\\
 &\,\,=\frac{\theta\,\alpha}{2}\,M_{{0,3}}+ \left( 3\,a_{{21}}+\frac{3\,\theta\,\alpha}{2}
 \right) M_{{1,2}}+a_{{3}}M_{{1,1}}.
 \end{split}
\end{equation}
For $(n,m)=(0,4),\,(p,q)=(0,0)$ and $f^{{\mathbbm{2}}}=\ind_{E^{*}\times E^{*}\times E^{*}\times E^{*}}$ in (\ref{eq-rev}), we have
\begin{equation}\label{E_{45}}
\begin{split}
 &-4\,u_{\mathbbm 1}M_{{
1,3}}+ \left( 6\,a_{{211}}+3
\,a_{{22}}+4\,a_{{31}}+a_{{4}}+2\,\theta+4\,u_{\mathbbm 1} \right) M_{{0,4}}\\
&\,\,=\left( 2\,\theta\,\alpha+6\,a_{{211}} \right) M_{{0,3}}+ \left( 3\,a_{{
22}}+4\,a_{{31}} \right) M_{{0,2}}+a_{{4}}\,\alpha.
\end{split}
\end{equation}
We solve for $M_{4,0},\,M_{3,1},\,M_{2£¬2},\,M_{1,3},\,M_{0,4}$ uniquely from (\ref{E_{41}})-(\ref{E_{45}}) which is also a system of linear equations with invertible coefficient matrix.
Substituting $(n,m)=(1,1),\,(p,q)=(2,0)$ and $f^{{\mathbbm{1}}}=f^{{\mathbbm{2}}}=\ind_{E^{*}},\,g^{{\mathbbm{1}}}=\ind_{E^{*}\times E^{*}}$, we have
\begin{equation}\label{F_{1}}
\begin{split}
\frac{\,\theta\,\alpha}{2}\,M_{{3,0}}-\(a_{{2}}+\frac{\theta\,
\alpha}{2}\)\,M_{{2,1}}-u_{\mathbbm 2}M_{{2,2}}+u_{\mathbbm 1}M_{{4,0}}+\(a_2-u_{\mathbbm 1}+u_{\mathbbm 2}\)M_{{3,1}}=0.
\end{split}
\end{equation}
Substituting $(n,m)=(2,1),\,(p,q)=(1,0)$ and $f^{{\mathbbm{1}}}=\ind_{E^{*}\times E^{*}},\,f^{{\mathbbm{2}}}=g^{{\mathbbm{1}}}=\ind_{E^{*}}$, we have
\begin{equation}\label{F_{2}}
\begin{split}
 \left( a_{{2}}+\frac{\theta\,\alpha}{2} \right) M_{{2,1}}+\frac{\theta\,
\alpha}{2}\,M_{{3,0}}- \left( \theta+u_{\mathbbm 2}+u_{\mathbbm 1}+a_{{2}} \right) M_{{3,1}}+u_{\mathbbm 2}M_{{2,2}}
+u_{\mathbbm 1}M_{{4,0}}=0.
\end{split}
\end{equation}
Substituting all the joint moments, the numerator of $$(\ref{F_{1}})-2\times(\ref{F_{2}})=0$$  is equivalent to
\begin{equation}\label{eq-contra}
\begin{split}
& \left( \theta+4\,u_{\mathbbm{1}} \right) \theta\,u_{\mathbbm{1}}a_{{3}} \left( -8\,a_{{4}}u_{\mathbbm{1}}-4\,a_
{{4}}a_{{3}}-4\,\theta\,a_{{4}}+8\,{a_{{3}}}^{2}+22\,a_{{3}}u_{\mathbbm{1}}+11\,a_{{
3}}\theta+8\,{a_{{21}}}^{2}+16\,a_{{21}}a_{{3}}\right.\\
&\left.\,\,\,\,\,\,\,\,\,\,\,\,+12\,a_{{211}}a_{{3}}-4
\,a_{{4}}a_{{21}}+10\,a_{{21}}u_{\mathbbm{1}}+5\,a_{{21}}\theta+3\,a_{{211}}\theta+6
\,a_{{211}}u_{\mathbbm{1}}  +12\,a_{{211}}a_{{21}}\right)=0.
\end{split}
\end{equation}
The term in the second parentheses  is always positive because
$-8a_4u_{\mathbbm{1}}+22a_3u_{\mathbbm{1}}\geq0$, $-4a_4a_3+8a_3^2\geq0$, $-4\theta a_4+11a_3\theta\geq0$
and
$-4a_4a_{21}+16a_{21}a_{3}\geq0.$
Since $u_{\mathbbm{1}}>0$ and $\theta>0$, a necessary condition for  (\ref{eq-contra}) is that $a_3=0$. By the consistent condition of coalescent rates
\begin{equation*}
\begin{cases}
a_2=a_{21}+a_3;\\
a_3=a_{31}+a_4;\\
a_{21}=a_{211}+a_{22}+a_{31},
\end{cases}
\end{equation*}
$a_3=0$ implies that  $a_4=a_{31}=a_{22}=0$ and $a_{2}=a_{21}=a_{211}=:a>0$. Substituting these values into the numerator of (\ref{F_{2}}), we get
$4a^3=0$ which contradicts $a>0$. Therefore, the process is not reversible.

A Maple note for all the calculations carried out in this paper is available upon request.
\end{appendices}

\renewcommand{\bibname}{References}
\bibliographystyle{plainnat}


\begin{thebibliography}{23}
\providecommand{\natexlab}[1]{#1}
\providecommand{\url}[1]{\texttt{#1}}
\expandafter\ifx\csname urlstyle\endcsname\relax
  \providecommand{\doi}[1]{doi: #1}\else
  \providecommand{\doi}{doi: \begingroup \urlstyle{rm}\Url}\fi

\bibitem[Berg et~al.(1984)Berg, Christensen, and Ressel]{Berg}
C.~Berg, J.~Christensen, and P.~Ressel.
\newblock \emph{Harmonic Analysis on Semigroups}.
\newblock Springer-Verlag New York Inc., 1984.

\bibitem[Birkner et~al.(2009)Birkner, Blath, M$\ddot{\text{o}}$hle,
  Steinr$\ddot{\text{u}}$cken, and Tams]{MJM}
M.~Birkner, J.~Blath, M.~M$\ddot{\text{o}}$hle, M.~Steinr$\ddot{\text{u}}$cken,
  and J.~Tams.
\newblock A modified lookdown construction for the {X}i-{F}leming-{V}iot with
  mutation and populations with recurrent bottlenecks process.
\newblock \emph{ALEA. Latin American Journal of Probability and Mathematical
  Statistics}, 6:\penalty0 25--61, 2009.


\bibitem[Donnelly and Kurtz(1999)]{DK99b}
P.~Donnelly and T.~G. Kurtz.
\newblock Particle representations for measure-valued population models.
\newblock \emph{The Annals of Probability}, 27:\penalty0 166--205, 1999.


\bibitem[Evans(1997)]{Eva97}
S.~N. Evans.
\newblock Coalescing markov labelled partitions and a continuous sites genetics
  model with infinitely many types.
\newblock \emph{Annales de l'Institut Henri Poincar\'e Probabilit\'es et
  Statistiques}, 33:\penalty0 339--358, 1997.

\bibitem[Feng et~al.(2011)Feng, Schmuland, Vaillancourt, and Zhou]{FZ}
S.~Feng, B.~Schmuland, J.~Vaillancourt, and X.~Zhou.
\newblock Reversibility of interacting {F}leming-{V}iot processes with mutation
  selection, and recombination.
\newblock \emph{Canadian Journal of Mathematics}, 63:\penalty0 104--122, 2011.

\bibitem[Fleming and Viot(1979)]{FV_79}
W.~H. Fleming and M.~Viot.
\newblock Some measure-valued markov processes in population genetics theory.
\newblock \emph{Indian University Mathematics Journal}, 28:\penalty0 817--843,
  1979.

\bibitem[Handa(1990)]{Handa_2}
K.~Handa.
\newblock A measure-valued diffusion process describing the stepping stone
  model with infinitely many alleles.
\newblock \emph{Stochastic Processes and their Applications}, 36:\penalty0
  269--296, 1990.

\bibitem[Kermany et~al.(2008)Kermany, Hickey, and Zhou]{AXD}
A.~R.~R. Kermany, D.~A. Hickey, and X.~Zhou.
\newblock Joint stationary moments of a two-island diffusion model of
  population subdivision.
\newblock \emph{Theoretical Population Biology}, 74:\penalty0 226--232, 2008.


\bibitem[Li et~al.(1999)Li, Shiga, and Yao]{LSY99}
Z.~Li, T.~Shiga, and L.~Yao.
\newblock A reversibility problem for {F}leming-{V}iot processes.
\newblock \emph{Electronic Communications in Probability}, 4:\penalty0 65--76,
  1999.

\bibitem[Li et~al.(2013)Li, Liu, Xiong, and Zhou]{LLXZ}
Z.~Li, H.~Liu, J.~Xiong, and X.~Zhou.
\newblock The reversibility and an {SPDE} for the generalized {F}leming-{V}iot
  processes with mutation.
\newblock \emph{Stochastic Processes and their Applications}, 123:\penalty0
  4129--4155, 2013.

\bibitem[Notohara(1990)]{Notohara}
M.~Notohara.
\newblock The coalescent and the genealogical process in geographically
  structured population.
\newblock \emph{Journal of Mathematical Biology}, 29:\penalty0 59--75, 1990.

\bibitem[Pitman(1999)]{Pitman99}
J.~Pitman.
\newblock Coalescents with multiple collisions.
\newblock \emph{The Annals of Probability}, 27:\penalty0 1870--1902, 1999.

\bibitem[Sagitov(1999)]{SS_99}
S.~Sagitov.
\newblock The general coalescent with asynchronous mergers of ancestral lines.
\newblock \emph{Journal of Applied Probability}, 36:\penalty0 1116--1125, 1999.

\bibitem[Sagitov(2003)]{SS_03}
S.~Sagitov.
\newblock Convergence to the coalescent with simultaneous multiple mergers.
\newblock \emph{Journal of Appliel Probability}, 40:\penalty0 839--854, 2003.

\bibitem[Schweinsberg(2000)]{Jason-S}
J.~Schweinsberg.
\newblock Coalescents with simultaneous multiple collisions.
\newblock \emph{Electronic Journal of Probability}, 5:\penalty0 1--50, 2000.

\bibitem[Shiga(1980{\natexlab{a}})]{Shiga801}
T.~Shiga.
\newblock An interacting system in population genetics.
\newblock \emph{Journal of Mathematics of Kyoto University}, 20:\penalty0
  213--242, 1980{\natexlab{a}}.

\bibitem[Shiga(1980{\natexlab{b}})]{Shiga802}
T.~Shiga.
\newblock An interacting system in population genetics. ii.
\newblock \emph{Journal of Mathematics of Kyoto University}, 20:\penalty0
  723--733, 1980{\natexlab{b}}.

\bibitem[Shiga(1982)]{Shiga82}
T.~Shiga.
\newblock Continuous time multi-allelic stepping stone models in population
  genetcis.
\newblock \emph{Journal of Mathematics of Kyoto University}, 22:\penalty0
  1--40, 1982.

\bibitem[Shiga and Uchiyama(1986)]{ShigaK}
T.~Shiga and K.~Uchiyama.
\newblock Stationary states and their stability of the stepping stone model
  involving mutation and selection.
\newblock \emph{Probability Theory and Related Fields}, 73:\penalty0 87--117,
  1986.

\end{thebibliography}
\end{document}